\documentclass[a4paper,12pt]{amsart}
\usepackage{mdwlist,combelow}
\usepackage{hyperref,enumerate,tasks,multicol,longtable,amssymb,amsmath,amscd,xcolor,mathrsfs,wasysym,graphicx,verbatim,mdframed}
\usepackage[all]{xy}
\usepackage{geometry}
\usepackage[utf8]{inputenc}
\usepackage[T1]{fontenc}
\RequirePackage{fix-cm}

\newtheorem{theorem}{Theorem}
\newtheorem*{claim*}{Claim}

\newtheorem*{theorema}{Theorem A}
\newtheorem*{theoremb}{Theorem B}
\newtheorem*{theoremc}{Theorem C}
\newtheorem*{theoremd}{Theorem D}
\newtheorem*{theoreme}{Theorem E}
\newtheorem*{theoremf}{Theorem F}

\newtheorem{definition}{Definition}
\newtheorem*{maindefinition}{Main definition}
\newtheorem{example}{Example}
\newtheorem{lemma}{Lemma}
\newtheorem{criterion}{Criterion}
\newtheorem*{lemma*}{Lemma}

\newtheorem{proposition}{Proposition}
\newtheorem*{proposition*}{Proposition}
\newtheorem{remark}{Remark}

\newcounter{conditions1}
\newcounter{conditions2}

\definecolor{light-gray}{gray}{0.95} %Just one number between 0 (black) and 1 (white), so 0.95 will be very light gray, 0.30 will be dark gray.
\definecolor{orange}{RGB}{255,127,0} %Three numbers given in the form red,green,blue; the quantity of each color is represented with a number between 0 and 255.
\definecolor{orange}{cmyk}{0,0.5,1,0} %Four numbers given in the form cyan,magenta,yellow,black; the quantity of each color is represented with a number between 0 and 1.
\definecolor{blue}{RGB}{0,0,255}

\newsavebox{\fminipagebox}
\NewDocumentEnvironment{fminipage}{m O{\fboxsep}}
 {\par\kern#2\noindent\begin{lrbox}{\fminipagebox}
  \begin{minipage}{#1}\ignorespaces}
 {\end{minipage}\end{lrbox}%
  \makebox[#1]{%
    \kern\dimexpr-\fboxsep-\fboxrule\relax
    \fbox{\usebox{\fminipagebox}}%
    \kern\dimexpr-\fboxsep-\fboxrule\relax
  }\par\kern#2
 }

\begin{document}
\title{Dirac products and concurring Dirac structures}
\author{Pedro Frejlich}
\address{UFRGS, Departamento de Matem\'atica Pura e Aplicada, Porto Alegre, Brasil}
\email{frejlich.math@gmail.com}
\author{David Mart\'inez Torres}
\address{Universidad Politécnica de Madrid, ETSAM section}
\email{df.mtorres@upm.es}

\begin{abstract}
 We discuss in this note two dual canonical operations on Dirac structures $L$ and $R$ --- the \emph{tangent product} $L \star R$ and the \emph{cotangent product} $L \circledast R$. Our first result gives an explicit description of the leaves of $L \star R$ in terms of those of $L$ and $R$, surprisingly ruling out the pathologies which plague general ``induced Dirac structures''.
 
 In contrast to the tangent product, the more novel contangent product $L \circledast R$ need not be Dirac even if smooth. When it is, we say that $L$ and $R$ \emph{concur}. 
 Concurrence captures commuting Poison structures, refines the \emph{Dirac pairs} of Dorfman and Kosmann-Schwarzbach, and it is our proposal as the natural notion of ``compatibility'' between Dirac structures. 
 
 The rest of the paper is devoted to illustrating the usefulness of tangent- and cotangent products in general, and the notion of concurrence in particular. Dirac products clarify old constructions in Poisson geometry, characterize Dirac structures which can be pushed forward by a smooth map, and mandate a version of a local normal form. Magri and Morosi's $P\Omega$-condition and Vaisman's notion of two-forms complementary to a Poisson structures are found to be instances of concurrence, as is the setting for the Frobenius-Nirenberg theorem. We conclude the paper with an interpretation in the style of Magri and Morosi of generalized complex structures which concur with their conjugates.
\end{abstract}

\maketitle

\section{Introduction}

\emph{Dirac structures}\footnote{We refer the reader to \cite{Gualtieri,Bursztyn,Meinrenken} for introductions to Dirac geometry.} are simultaneous generalizations of closed two-forms, foliations and Poisson structures. Heuristically, a Dirac structure $L$ on a manifold $M$ is a smooth partition of $M$ into initial submanifolds, 
%(the \emph{leaves} of $L$)
on which there are induced, smoothly varying, closed two-forms. More formally, a Dirac structure $L$ is a subbundle of the \emph{generalized tangent bundle} $\mathbb{T} M:=TM \oplus T^*M$, which is Lagrangian for the canonical, symmetric bilinear pairing
 \begin{align}
\langle\cdot,\cdot\rangle:\mathbb{T} M \times \mathbb{T} M \longrightarrow \mathbb{R}, \quad \langle u+\xi,v+\eta\rangle:=\iota_v\xi+\iota_u\eta,
 \end{align}and whose space of sections $\Gamma(L)$ is involutive under the \emph{Dorfman bracket}:
  \begin{align}
[u+\xi,v+\eta]:=[u,v]+\mathscr{L}_u\eta - \iota_v\mathrm{d}\xi.
 \end{align}A Dirac structure $L$ is in particular a \emph{Lie algebroid} --- that is, $[\cdot,\cdot]$ turns $\Gamma(L)$ into a Lie algebra, and the restriction of the canonical projection $\mathrm{pr}_T:\mathbb{T} M \to TM$ to $L$ defines a homomorphism of Lie algebras from $\Gamma(L)$ into the space of vector fields. 
 The tangent space $T_xS$ to a leaf $S$ of this Lie algebroid is $\mathrm{pr}_T(L_x)$, and the leaf inherits a closed two-form given by $\omega_S(u,v) = \langle u+\xi,v\rangle$ for any representative $u+\xi \in L$.
 
Since their inception, Dirac structures have been remarkably useful in successfully tackling problems in pure Poisson geometry. Part of it may be a manifestation of Weinstein's "everything is a Lagrangian" metaprinciple \cite{Wein_Everything}, but, more concretely, it stems from the fact that, very often, the methods and models required to address a purely Poisson-theoretic problem lie outside the realm of Poisson geometry proper (see \cite{PT1,Meinrenken} for some manifestations of this). An instance of this phenomenon concerns the various circumstances under which a submanifold $X$ of a Poisson manifold $(M,\pi)$ ``inherits'' a geometric structure from its ambient manifold. Now, while we cannot in general speak of the ``pullback'' of the Poisson structure by the inclusion $i:X \to M$, the \emph{pointwise} notion of what ``induced'' ought to mean is mandated. By pointwise we mean that any submanifold has an induced \emph{Lagrangian family} $i^!(\mathrm{Gr}(\pi)) \subset \mathbb{T}X$ --- that is, a subspace that meets every fibre in a Lagrangian vector space; its smoothness is not ensured, but, should $i^!(\mathrm{Gr}(\pi))$ be smooth, then it is automatically a Dirac structure. 
\vspace{0.3cm}

The goal of the present note is to discuss two canonical  operations on Dirac structures: the \emph{tangent-} and the \emph{cotangent products}
\begin{align*}
& (L,R) \mapsto L \star R, & (L,R) \mapsto L \circledast R
\end{align*}
of two Dirac structures $L,R \subset \mathbb{T}M$. Truth be told, the operations $\star$ and $\circledast$ are actual products (as opposed to relations) only on the space $\mathrm{Lag}(\mathbb{T} M)$ of Lagrangian families\footnote{This is a recurring theme in the theory of Dirac structures: it is often best to decouple the linear algebraic requirements from the differential geometric ones.}:
\begin{align*}
\star,\circledast:\mathrm{Lag}(\mathbb{T} M) \times \mathrm{Lag}(\mathbb{T} M) \longrightarrow \mathrm{Lag}(\mathbb{T} M).
\end{align*}

Tangent products have been studied in the literature before, see \cite[Section 2.3]{ABM} and \cite[Section 2.5]{Gualtieri}, the latter under the name \emph{tensor products}. Our choice of terminology reflects the fact, as we will explain below, that these products are algebraically dual to one another, in a manner directly analogous to the duality between $TM$ and $T^*M$. A crucial feature of the tangent product is that it has a ``geometric construction'', namely: the tangent product $L \star M$ is nothing but the pullback under the diagonal map $\Delta:M \to M \times M$ of the product Dirac structure $L \times R$:
\begin{align*}
 L \star R = \Delta^!(L \times R).
\end{align*}
 Our first main result concerns the (in general delicate) description of the leaves of tangent products:

\begin{theorema}
If the tangent product of two Dirac structures is again Dirac, then the leaves of one structure meet the leaves of the other cleanly, and (the connected components of) the intersections of leaves are the leaves of the tangent product.
\end{theorema}
In the description of the tangent product as a pullback/induced Dirac structure one would surmise that $L \star R$ would be subject to the recently discovered ``jumping phenomenon'' \cite{BFT} --- in which a leaf of the Dirac structure induced on a submanifold need \emph{not} be entirely contained in a leaf of the ambient manifold. Theorem A conveys the rather pleasant surprise that the tangent product ``filters out'' non-clean intersections. (See Section \ref{sec : Cleanness of tangent products} for more details.)

In any case, the algebraic symmetry between tangent and cotangent products is broken when we take into account the involutivity condition: while the smoothness of the tangent product $L \star R$ of two Dirac structures is enough to ensure that it is a Dirac structure, this is no longer true for the cotangent product $L \circledast R$, which need not be Dirac even if it is smooth. This in fact brings us to the core message of the paper: 

\begin{minipage}{0.95\textwidth}
  {\small
\begin{mdframed}[backgroundcolor=olive!10]
The natural notion of ``compatibility'' between Dirac structures $L$ and $R$ is the condition that $L \circledast R$ itself be a Dirac structure, in which case we say that $L$ and $R$ {\bf concur}.
\end{mdframed}
}
\end{minipage}

\vspace{0.2cm}

A different notion of ``compatible Dirac structures'' --- namely, that of \emph{Dirac pairs} --- had been proposed in the seminal work of I. Dorfman \cite{Dorfman_book}, and developed by Y. Kosmann-Schwarzbach \cite{Kosmann}. Our proposal is a refinement of theirs:
\begin{theoremb}
 Concurring Dirac structures form a Dirac pair,
\end{theoremb}
\noindent and we will argue that the notion of concurrence is the ``canonical'' one in the Dirac context. When both $L$ and $R$ correspond to Poisson structures, concurrence and being Dirac pairs coincide, and select commuting Poisson structures, but they already differ if $L$ and $R$ correspond to foliations: in that case, our proposal boils down to requiring that the sum of their tangent bundles be the tangent bundle to a foliation. This motivates the language\footnote{From Webster's Revised Unabridged Dictionary of the English Language of 1913:
\begin{minipage}{0.95\textwidth}
Con.cur, v. i. imp. \& p. p. Concurred (?); p. pr. \& vb. n. Concurring. L. \emph{concurrere} to run together, agree; \emph{con-} + \emph{currere} to run. See Current.
\begin{multicols}{2}
1. To run together; to meet. \emph{Obs.} ``Anon they fierce encountering both concurred/With grisly looks and faces like their fates.'' J. Hughes. \ 2. To meet in the same point; to combine or conjoin; to contribute or help toward a common object or effect. ``When outward causes concur.'' Jer. Colier. \ 3. 
To unite or agree (in action or opinion); to join; to act jointly; to agree; to coincide; to correspond. ``Mr. Burke concurred with Lord Chatham in opinion.'' Fox. ``Tories and Whigs had concurred in paying honor to Walker.'' Makaulay. ``This concurs directly with the letter.'' Shak. \ 4. To assent; to consent. \emph{Obs.} Milton.
\end{multicols}
Syn. -- To agree; unite; combine; conspire; coincide; approve; acquiesce; assent. 
\end{minipage}
}
we employ.

The remainder of the paper is devoted to illustrating not only the usefulness of the tangent and cotangent products as canonical tools which unify and clarify various phenomena, but also how the notion of concurrence had already appeared in the literature under various guises. 

The formulation of Libermann's theorem in Dirac language \cite{Frejlich_Marcut} is one of the motivating examples underpinning the notion of concurrence: a Dirac structure on the total space of a simple foliation pushes forward to a Dirac structure on its leaf space exactly when it concurs with the Dirac structure corresponding to the foliation. This interpretation of concurrence with foliations leads to an pleasantly ``obvious'' version of the local normal form for Dirac structures of \cite{Blohmann}:
\begin{theoremc}
 In the neighborhood of any point, a Dirac structure is the gauge transform of the pullback of a Poisson structure by a surjective submersion.
\end{theoremc}
One previous manifestation of concurrence in the literature was in the guise of the classical $P\Omega$ condition proposed by Magri and Morosi in the context of bihamiltonian structures \cite{Magri_Morosi,KS_Magri}; it  turns out to be also an instance of Vaisman's notion of two-forms complementary to a Poisson structure \cite{Vaisman}:
\begin{theoremd}
 A Poisson structure and a closed two-form concur exactly when they form a $P\Omega$-structure; this is also equivalent the the two-form being complementary to the Poisson structure.
\end{theoremd}
The scalar extension to complex numbers of the notions of tangent- and cotangent products is immediate, and can be applied to the complex Dirac structures of \cite{Aguero_Rubio}. We illustrate our methods in two cases: those of:
\begin{itemize}
 \item \emph{involutive structures}, --- that is, complex subbundles of $TM \otimes \mathbb{C}$ whose space of sections is involutive under the scalar extension of the Lie bracket of vector fields;
 \item \emph{generalized complex structures}, --- that is, complex Dirac structures $L \subset \mathbb{T}M \otimes \mathbb{C}$ which are transverse to their conjugates.
\end{itemize}

\begin{theoreme}
 The Frobenius-Nirenberg theorem \cite{Nirenberg} is valid exactly for those involutive structures $E \subset TM \otimes \mathbb{C}$ which concur with their conjugates.
\end{theoreme}

\begin{theoremf}
 The generalized complex structures which concur with their conjugates are exactly those which simultaneously induce $PN$- and $\Omega N$-structures in the sense of Magri-Morosi \cite{Magri_Morosi}.
\end{theoremf}

\subsection*{Acknowledgements}
We would like to express our gratitude to Ioan M\u{a}rcu\c{t}, whose early involvement in this project was invaluable, and who contributed to the development of the paper through many insightful discussions. We would also like to thank Dan Ag\"uero and Igor Mencattini for useful feedback on an earlier version of the paper, and the anonymous referee, whose suggestions and corrections greatly improved the manuscript.

\section{Lagrangian families and Dirac structures}\label{sec : Lagrangian stuff}

In this preparatory section, we collect basic facts about the generalized tangent bundle and Dirac structures which will be used throughout the paper. For a more thorough introduction, we recommend \cite{Bursztyn}, \cite{Gualtieri}, \cite{Marcut} and \cite{Meinrenken}.\\

The {\bf generalized tangent bundle} $\mathbb{T} M:=TM \oplus T^*M$ of a smooth manifold $M$ is equipped with the canonical, symmetric bilinear pairing
 \begin{align}\label{eq : symmetric bilinear pairing}
\langle\cdot,\cdot\rangle:\mathbb{T} M \times \mathbb{T} M \longrightarrow \mathbb{R}, \quad \langle u+\xi,v+\eta\rangle:=\iota_v\xi+\iota_u\eta,
 \end{align}and the {\bf Dorfman bracket} on its space of sections $\Gamma(\mathbb{T}M)$, given by
 \begin{align}\label{eq : dorfman bracket}
[u+\xi,v+\eta]:=[u,v]+\mathscr{L}_u\eta - \iota_v\mathrm{d}\xi;
 \end{align}
 here $[u,v]$ denotes the Lie bracket of vector fields $u,v \in \mathfrak{X}(M)$, $\mathrm{d}:\Omega(M) \to \Omega(M)$ the de Rham differential, and $\mathscr{L}_u$ the Lie derivative by the vector field $u$. The Dorfman bracket is a {\bf Leibniz bracket}, in the sense that
\begin{align*}
 & [a,[b,c]] = [[a,b],c]+[b,[a,c]], & a,b,c \in \Gamma(\mathbb{T}M),
\end{align*}but it is not a Lie bracket since it is not antisymmetric:
\begin{align*}
 & [a,b]+[b,a]=\mathrm{d}\langle a,b\rangle, & a,b \in \Gamma(\mathbb{T}M).
\end{align*}

A {\bf linear family} $L \subset \mathbb{T}M$ is a subset which meets each fibre $\mathbb{T}_xM$ in a vector space. We stress that we assume no continuity for $x \mapsto L \cap \mathbb{T}_xM$; for example,
\begin{align*}
 & L_x = \begin{cases}
          T_xM & \text{for} \ x \neq 0;\\
          T^*_xM & \text{for} \ x =0
         \end{cases}, 
 & R_x = \begin{cases}
          0 & \text{for} \ x \neq 0;\\
          T_xM & \text{for} \ x =0
         \end{cases}
\end{align*}are perfectly well-defined linear families on $M=\mathbb{R}$. The {\bf orthogonal} $L^{\perp}$ of a linear family $L$ is the linear family
\begin{align*}
 L^{\perp} = \{ a \in \mathbb{T}M \ | \ \langle a, \cdot\rangle|_L = 0\}.
\end{align*}
Note that, since \eqref{eq : symmetric bilinear pairing} is nondegenerate, $2\dim(M)=\mathrm{rk}(L)+\mathrm{rk}(L^{\perp})$. A {\bf Lagrangian} family is a linear family $L$ which equals its orthogonal, $L=L^{\perp}$. Observe that, in this case, $\dim(M)=\mathrm{rk}(L)$. We collect all Lagrangian families on $M$ into a set $\mathrm{Lag}(\mathbb{T}M)$. Given a smooth map $f:M \to N$, we say that
\begin{align*}
 && a \in \mathbb{T}M && \text{and} && b \in \mathbb{T}N
\end{align*}are {\bf $f$-related} if
\begin{align*}
 & f_*\mathrm{pr}_T(a)=\mathrm{pr}_T(b), & \mathrm{pr}_{T^*}(a) = f^*\mathrm{pr}_{T^*}(b),
\end{align*}in which case we write $a \sim_f b$. This induces pointwise maps of {\bf pushforward} and {\bf pullback},
\begin{align*}
 & f_! : \mathrm{Lag}(\mathbb{T}_xM) \to \mathrm{Lag}(\mathbb{T}_{f(x)}N), & f^! : \mathrm{Lag}(\mathbb{T}_{f(x)}N) \to \mathrm{Lag}(\mathbb{T}_xM),
\end{align*}
given by
\begin{align*}
 && f_!(L) = \{ b \in \mathbb{T}_{f(x)}N \ | \ \exists \ a \in L, \ a \sim_f b\}, && f^!(R) = \{ a \in \mathbb{T}_{x}M \ | \ \exists \ b \in R, \ a \sim_f b\}.
\end{align*}
If $L \in \mathrm{Lag}(\mathbb{T}M)$ and $R \in \mathrm{Lag}(\mathbb{T}N)$, we say that
\begin{align*}
 f : (M,L) \to (N,R)
\end{align*}is {\bf forward} if $f_!(L)=R$, and {\bf backward} if $L=f^!(R)$.

A linear family $L$ is {\bf smooth} if it is a vector subbundle of $\mathbb{T}M$, and we refer to a smooth Lagrangian family $L$ as a {\bf Lagrangian subbundle}. For a Lagrangian subbundle $L$, the parity of $\mathrm{rank}\mathrm{pr}_{TM}(L_x)$ is locally independent of $x \in M$; in particular, on a connected manifold $M$, Lagrangian subbundles are either {\bf even} or {\bf odd} (see \cite[Section 1.4]{ABM} or \cite[Section 1.1]{Gualtieri}). A Lagrangian subbundle $L$ is a {\bf Dirac structure} if $\Gamma(L)$ is involutive under the Dorfman bracket \eqref{eq : dorfman bracket}; this condition is equivalent to requiring that the {\bf Courant tensor}
\begin{align*}
 & \Upsilon_L \in \Gamma(\wedge^3L^*), & \Upsilon_L(a,b,c):=\langle [a,b],c\rangle.
\end{align*}
vanish identically. 

If $L$ is a Dirac structure on $M$, and $f:N \to M$ is a smooth map, the {\bf pullback Lagrangian family} $f^!(L)$ is a Lagrangian family on $N$. In contrast, if $f:M \to N$ is a smooth map, the pointwise pushforward $x \mapsto f_!(L_x)$ is in principle only a Lagrangian family in the pullback $f^*(\mathbb{T}N)$; it can be regarded as a {\bf pushforward Lagrangian family} on $N$ exactly when $L$ is $f$-invariant:
\begin{align*}
 && f(x)=f(y) && \Longrightarrow && f_!(L_x) = f_!(L_y),
\end{align*}
It should be stressed that, in general, neither pullbacks nor pushforwards of Dirac structures need be smooth. However:
\begin{lemma}\label{lem : pullbacks or pushforwards of Dirac}
 If the pullback or the pushforward of a Dirac structure is smooth, then it is itself a Dirac structure.
\end{lemma}
\begin{proof}
 The pullback statement is proved in \cite[Proposition 5.6]{Bursztyn} and the pushforward statement is proved in \cite[Proposition 5.9]{Bursztyn}.
\end{proof}

If $L \subset \mathbb{T}M$ is a Dirac structure, the restriction of the Dorfman bracket to $\Gamma(L)$ turns it into a Lie algebra, equipping $L$ with the structure of a {\bf Lie algebroid}, with anchor $\mathrm{pr}_T:L \to TM$. There is then an induced partition of $M$ into initial submanifolds, the {\bf leaves} of $L$, which can be described as follows: the leaf $\mathrm{S}_L(x)$ of $L$ through $x \in M$ consists of all endpoints $c(1)$ of smooth curves $c:[0,1] \to M$ which start at $x$, and can be covered by an {\bf $L$-path} $a:[0,1] \to L$ --- that is, a smooth path such that
\begin{align*}
 & c = p(a), & \dot{c}=\mathrm{pr}_T(a),
\end{align*}where $p:L \to M$ denotes the canonical projection. The inclusion $i:\mathrm{S}_L(x) \to M$ gives a backward map
\begin{align*}
 i : (\mathrm{S}_L(x),\omega_{\mathrm{S}_L(x)}) \to (M,L),
\end{align*}where $\omega_{\mathrm{S}_L(x)} \in \Omega^2(\mathrm{S}_L(x))$ is the closed two-form given by
\begin{align*}
 && \omega_{\mathrm{S}_L(x)}(u,v):=\langle a,v\rangle, && a \in L, && u=\mathrm{pr}_T(a).
\end{align*}

\begin{remark}\normalfont
The abelian groups of two-forms $\Omega^2(M)$, non-zero real numbers $\mathbb{R}^{\times}$ and bivectors $\mathfrak{X}^2(M)$ act on $\mathbb{T}M$ via
\begin{enumerate}
 \item $\Omega^2(M) \curvearrowright \mathbb{T}M$, $\mathcal{R}_{\omega}(a):=a+\omega^{\sharp}(\mathrm{pr}_T(a))$;
 \item $\mathbb{R}^{\times} \curvearrowright \mathbb{T}M$, $\mathcal{R}_{t}(a):=t\mathrm{pr}_T(a)+\mathrm{pr}_{T^*}(a)$;
 \item $\mathfrak{X}^2(M) \curvearrowright \mathbb{T}M$, $\mathcal{R}_{\pi}(a):=a+\pi^{\sharp}(\mathrm{pr}_{T^*}(a))$,
\end{enumerate}and induce actions of these groups on the set of Lagrangian families on $M$. The actions 1-3 all restrict to actions on the set of Lagrangian subbundles. The action 2 preserves the set of Dirac structures, and 1 preserves the set of Dirac structures when the two-form is \emph{closed}.
\end{remark}

\section{Tangent products}\label{sec : Tangent product}

The action of two-forms on Lagrangian families by gauge transformation,
\begin{align*}
 & \mathcal{R} : \Omega^2(M) \times \mathrm{Lag}(\mathbb{T}M) \to \mathrm{Lag}(\mathbb{T}M), & (\omega,L) \mapsto \mathcal{R}_{\omega}(L)
\end{align*}
has a canonical extension to a function
\begin{align*}
 & \star : \mathrm{Lag}(\mathbb{T}M) \times \mathrm{Lag}(\mathbb{T}M) \to \mathrm{Lag}(\mathbb{T}M), & (L,R) \mapsto L \star R
\end{align*}
which is known under the name of \emph{tensor product} of Dirac structures (see e.g. \cite{ABM} and \cite{Gualtieri}). We will refer to it instead as the \emph{tangent product}, to highlight its duality with respect to the cotangent product to be introduced later. We begin by invoking \cite[Definition 2.20]{Gualtieri} in the context of Lagrangian families:

\begin{definition}
 Let $\varDelta : M \to M \times M$ denote the diagonal map. We refer to
 \begin{align*}
  & \star : \mathrm{Lag}(\mathbb{T}M) \times \mathrm{Lag}(\mathbb{T}M) \to \mathrm{Lag}(\mathbb{T}M), & L \star R:=\varDelta^!(L \times R)
 \end{align*}as the {\bf tangent product} of the Lagrangian families $L$ and $R$.
\end{definition}

It will be useful later to have a purely algebraic description of tangent products:
\begin{lemma}
 The tangent product of two Lagrangian families $L$ and $R$ on $M$ is given by
 \begin{align*}
  L \star R = \{ a+\mathrm{pr}_{T^*}(b) = \mathrm{pr}_{T^*}(a)+b \ | \ (a,b) \in L \times R, \ \mathrm{pr}_T(a-b)=0\}.
 \end{align*}
\end{lemma}
\begin{proof}
 Since $\varDelta_*(u) = (u,u)$ and $\varDelta^*(\xi,\eta) = \xi + \eta$, we see that, for
 \begin{align*}
  && u,u_L,u_R \in TM, && \xi,\xi_L,\xi_R \in T^*M,
 \end{align*}
 we have that
 \begin{align*}
  \mathbb{T} M \ni u+\xi  \sim_{\varDelta} (u_L+\xi_L,u_R+\xi_R) \in L \times R
 \end{align*}exactly when
 \begin{align*}
  && u_L=u=u_R && \text{and} && \xi = \xi_L+\xi_R,
 \end{align*}which is to say that $L \star R=\varDelta^!(L \times R)$ is given by the formula in the statement.
\end{proof}

\begin{lemma}\label{lem : tangent product}
 The map $\star$ defines an associative and commutative product on Lagrangian families, for which $TM$ acts as unity and $T^*M$ as zero:
  \begin{multicols}{2}
 \begin{enumerate}[a)]
 \setcounter{enumi}{\value{conditions1}}
 \item $L \star (R \star S) = (L \star R) \star S$;
 \item $L \star R=R \star L$;
 \setcounter{conditions1}{\value{enumi}}
\end{enumerate}
\begin{enumerate}[a)]
 \setcounter{enumi}{\value{conditions1}}
 \item $L \star TM = L$;
 \item $L \star T^*M=T^*M$.
 \setcounter{conditions1}{\value{enumi}}
\end{enumerate}
 \end{multicols}
\end{lemma}

\begin{example}\label{example : two-forms}\normalfont
 The tangent product of the graph $\mathrm{Gr}(\omega)$ of a two-form $\omega$ with any Lagrangian subbundle $L$ is smooth, and coincides with the gauge transformation of $L$ by $\omega$:
 \[\mathrm{Gr}(\omega) \star L = \mathcal{R}_{\omega}(L)= \{a+\iota_{\mathrm{pr}_{T}(a)}\omega\ | \ a \in L\}.\]In particular, $\mathrm{Gr}(\omega) \star \mathrm{Gr}(\omega')=\mathrm{Gr}(\omega+\omega')$.
\end{example}

\begin{lemma}\label{lem : behavior of star under pullbacks}
For all smooth maps $f : N \to M$ and Lagrangian families $L,R$ on $M$, we have that
\begin{align*}
 f^!(L \star R) = f^!(L) \star f^!(R).
\end{align*}
\end{lemma}
\begin{proof}
By definition of $\star$,
  \begin{align*}
   f^!(L \star R) & = f^!\varDelta_M^!(L \times R) = (\varDelta_M \circ f)^!(L \times R) = ((f,f) \circ \varDelta_N)^!(L \times R) = \varDelta_N^!(f^!(L) \times f^!(R)) \\ & =  f^!(L) \star f^!(R),
  \end{align*}where we denoted by $\varDelta_M$ and by $\varDelta_N$ the diagonal maps of $M$ and $N$.
\end{proof}

Because the tangent product of two Lagrangian families is given by a pullback construction, which need not preserve smoothness, it is to be expected that the tangent product of two Lagrangian subbundles need not be smooth, as the following example illustrates:

\begin{example}\normalfont
For all Lagrangian family $L$, we have that
\begin{align*}
 L \star \mathcal{R}_{-1}(L) = \mathrm{pr}_T(L) \oplus \mathrm{pr}_T(L)^{\circ}.
\end{align*}Therefore, for a Lagrangian subbundle $L$, $L \star \mathcal{R}_{-1}(L)$ is smooth exactly when $\mathrm{pr}_T : L \to TM$ has locally constant rank.
\end{example}

It is therefore useful to have a criterion at hand to ensure smoothness of tangent products:
\begin{criterion}\label{star criterion}
 For Lagrangian subbundles $L,R \subset \mathbb{T}M$, a sufficient condition for $L \star R$ to be smooth is that the vector bundle map
 \begin{align*}
  & \delta_T : L \oplus R \to TM, & \delta_T(a,b):=\mathrm{pr}_T(a-b)
 \end{align*}have locally constant rank. If that is the case, then
 \begin{align*}
  \Gamma(L \star R) = \{a+\mathrm{pr}_{T^*}(b) \ | \ a \in \Gamma(L), \ b \in \Gamma(R), \ \mathrm{pr}_T(a)=\mathrm{pr}_T(b)\}.
 \end{align*}
\end{criterion}
\begin{proof}
 Consider the linear map
 \begin{align*}
  & j : L \oplus R \to \mathbb{T}M, & j(a,b):=a+\mathrm{pr}_{T^*}(b).
 \end{align*}The Lagrangian family $L \star R$ is then
 \begin{align*}
  L \star R = j(\ker(\delta_T)).
 \end{align*}
 If $\delta_T$ has locally constant rank, the formula above expresses the Lagrangian family of constant rank $L \star R$ as the image of a smooth vector bundle map $j:\ker(\delta_T) \to \mathbb{T}M$, and this implies that $L \star R$ is a Lagrangian subbundle. Any section $c \in \Gamma(L \star R)$ lifts to a section $\widetilde{c} \in \Gamma(\ker(\delta_T))$, and setting
 \begin{align*}
  & a:=\mathrm{pr}_1(\widetilde{c}) \in \Gamma(L), & b:=\mathrm{pr}_2(\widetilde{c}) \in \Gamma(R),
 \end{align*}
 where $\mathrm{pr}_1:L \oplus R \to L$ and $\mathrm{pr}_2:L \oplus R \to R$, we see that $c = a+\mathrm{pr}_T(b)$.
 \end{proof}
 
 \begin{remark}\normalfont\label{rem : iterated star criterion}
  If $L_1,L_2,...,L_n$ are Lagrangian subbundles, one may apply Criterion \ref{star criterion} inductively to ensure that $L_1 \star \cdots \star L_n$ is smooth if the vector bundle maps
  \begin{align*}
   && \delta_T : L_1 \oplus L_2 \to TM, && \delta_T : (L_1 \star L_2) \oplus L_3 \to TM, && \cdots && \delta_T : (L_1 \star \cdots \star L_{n-1}) \oplus L_n \to TM
  \end{align*}
all have locally constant rank. In particular, the largest open set $U \subset M$ on which all of these maps have locally constant rank is dense in $M$.
 \end{remark}

\begin{example}\label{example : foliations}\normalfont
 If $E,F \subset TM$ are vector subbundles, and $L=\mathrm{Gr}(E)$, $R=\mathrm{Gr}(F)$ are the ensuing Lagrangian subbundles, then
 \begin{align*}
  L \star R & =(E \cap F)\oplus (E^{\circ} + F^{\circ})
 \end{align*}is smooth exactly when $E \cap F \subset TM$ is a vector subbundle.
\end{example}

\begin{proposition}\label{pro : automatic involutivity}
 The tangent product $L \star R$ of Dirac structures $L$ and $R$ on $M$ is Dirac, provided that $L \star R$ is a Lagrangian subbundle. 
\end{proposition}
\begin{proof}
 Because $L \star R = \varDelta^!(L \times R)$ is the pullback of a Dirac structure, by Lemma \ref{lem : pullbacks or pushforwards of Dirac}, $L \star R$ is a Dirac if it is smooth.
\end{proof}
It is convenient to recast this proof in an algebraic framework more suggestive of the arguments to come:
\begin{proof}[Second proof of Proposition \ref{pro : automatic involutivity}]
 We compute the Courant tensor $\Upsilon_{L \star R}$. Let $U$ be the largest open set in $M$ on which Criterion \ref{star criterion} applies to $L|_U$ and $R|_U$. Then $U$ is dense in $M$, and any three sections $a,b,c \in \Gamma(L|_U \star R|_U)$ can be written as
 \begin{align*}
  && a = u + \xi_L+\xi_R, && b = v + \eta_L+\eta_R, && c = w + \zeta_L+\zeta_R, 
 \end{align*}
 where
 \begin{align*}
  && a_L = u + \xi_L, && b_L = v + \eta_L, && c_L = w + \zeta_L && \text{lie in} \ \Gamma(L|_U), \\
  && a_R = u + \xi_R, && b_R = v + \eta_R, && c_R = w + \zeta_R && \text{lie in} \ ^\Gamma(R|_U).
 \end{align*}
Then
\begin{align*}
 \Upsilon_{L \star R}(a,b,c) & = \langle [a,b],c\rangle \\
 & = \langle [u + \xi_L+\xi_R,v + \eta_L+\eta_R],w + \zeta_L+\zeta_R\rangle \\
 & = \langle [u,v] + [\xi_L,v]+[\xi_R,v] + [u,\eta_L]+[u,\eta_R],w + \zeta_L+\zeta_R\rangle \\
 & = \langle [u,v],\zeta_L \rangle + \langle [\xi_L,v]+ [u,\eta_L],w\rangle + \langle [u,v],\zeta_R \rangle + \langle [\xi_R,v]+[u,\eta_R],w\rangle \\
 & = \Upsilon_{L}(a_L,b_L,c_L)+\Upsilon_{R}(a_R,b_R,c_R)\\
 & = 0
\end{align*}
shows that $\Upsilon_{L \star R}$ vanishes on the dense set $U$, and therefore vanishes on $M$ identically.
\end{proof}

\section{Cleanness of tangent products}\label{sec : Cleanness of tangent products}

It was recently pointed out in \cite{BFT} that, in general, submanifolds $i:X \to M$ which inherit a Dirac structure $i^!(L)$ from its ambient manifold $(M,L)$ may exhibit a \emph{jumping phenomenon}: a leaf of the induced Dirac structure does not necessarily lie in a single leaf of the ambient Dirac structure. This is somewhat surprising in the sense that the characteristic distribution $\mathrm{pr}_T(i^!(L))$ is always contained in $\mathrm{pr}_T(L)$; in fact,
\begin{align*}
 T_x\mathrm{S}_{i^!(L)}(x) = T_xX \cap T_x\mathrm{S}_{L}(x)
\end{align*}for all $x \in X$. The issue is that a $i^!(L)$-path $a_X:I \to i^!(L)$ need not come from an $L$-path $a:I \to L$; in fact, that is the case exactly when $X$ and $\mathrm{S}_{L}(x)$ ``meet cleanly''. 

One should then in principle expect that the tangent product $L \star R$ of Dirac structures $L$ and $R$ would display the same jumping phenomenon described above, since $L \star R$ arises precisely as the Dirac structure which (under favorable circumstances)
\begin{align*}
 (M,L) \times (M,R)
\end{align*}
induces on the diagonal submanifold $\varDelta(M) \subset M \times M$. In particular, it would be expected that the partition of leaves of $L \star R$ cannot be inferred from those of $L$ and $R$ alone. Our first result \underline{negates} that possibility: we show below that if the tangent product of two Dirac structures is again Dirac, then leaves of both structures meet cleanly, and have induced Dirac structures. In particular, the jumping phenomenon is \underline{ruled out} for tangent products. 

The rest of the section is devoted to proving this assertion. Some care must be taken, and a specific definition is required, in dealing with clean intersections of leaves, since the latter are not in general embedded submanifolds. The following definition is a relaxation of the notion of ``leaf-like submanifolds'' introduced in \cite[Appendix]{Frejlich_Marcut_homology} which is tailored to our needs:

\begin{definition}
 A subset $X \subset M$ is called a {\bf leafy submanifold} if the following condition is met: for each $x \in X$ there exists an open neighborhood $U$ of $x$ in $M$, with the property that $U \cap X$ has at most countably many\footnote{Observe that the countability hypothesis has the role of ensuring that the submanifold $X$ be second-countable.} connected components, each of which is an embedded submanifold of $M$.
\end{definition}

Leaf-like submanifolds are leafy submanifolds, and one can prove exactly as in \cite[Appendix]{Frejlich_Marcut_homology} that leafy submanifolds are initial submanifolds $i:X \to M$. So the smooth structure of $X$ (albeit not in general its topology) is induced from that of $M$. 

One of the advantages of the notion of leafy submanifold is that the condition of clean intersection is perfectly analogous to that between embedded submanifolds (cf. \cite[Definition A.2]{BFT}):

\begin{definition}\label{def : clean}
 Two leafy submanifolds $X, Y \subset M$ meet {\bf cleanly} if:
 \begin{enumerate}[a)]
  \item $X \cap Y \subset M$ is a leafy submanifold;
  \item $T(X \cap Y)=TX \cap TY$.
 \end{enumerate}
\end{definition}

The following lemma, adapted from \cite{BFT}, is main technical tool we will need in the sequel:

\begin{lemma}\label{lem : clean iff}
 The following conditions on leafy submanifolds $X, Y$ of a manifold $M$ are equivalent:
 \begin{enumerate}
  \item $X$ and $Y$ meet cleanly;
  \item Around each $z \in X \cap Y$ there is a chart $\varphi : M \supset U \to \mathbb{R}^m$ which maps the connected components
  \begin{align*}
   & X_z \ \text{of} \ U \cap X, & Y_z \ \text{of} \ U \cap Y
  \end{align*}
to the intersection of $\varphi(U)$ with vector subspaces $V_X,V_Y \subset \mathbb{R}^m$:
 \begin{align*}
  & \varphi(X_z) = \varphi(U) \cap V_X, & \varphi(Y_z) = \varphi(U) \cap V_Y.
 \end{align*}
  \item $X \cap Y$ is the disjoint union of (possibly uncountably many) leafy submanifolds $W \subset M$,
  \begin{align*}
   X \cap Y = \coprod_{W \in \Lambda(X,Y)} W,
  \end{align*}with the property that
  \begin{align*}
   && z \in W && \Longrightarrow && T_zW=T_zX_z \cap T_zY_z.
  \end{align*}
 \end{enumerate}
\end{lemma}
\begin{proof}
Observe that each of the conditions 1-3 is satisfied exactly when it is satisfied when localized on open subsets of $M$. If $x \in X \cap Y$ has an open neighborhood $U$ in $M$, for which
\begin{align*}
 & U \cap X = \coprod_{Z \in \Lambda(U,X)}Z, & U \cap Y = \coprod_{W \in \Lambda(U,Y)}W,
\end{align*}then $U \cap X$ meets $U \cap Y$ cleanly exactly when the (embedded) manifolds $Z$ and $W$ all meet cleanly. Applying \cite[Proposition 5.7]{BFT}, we deduce that the statement of the lemma holds on $U$ --- and therefore holds on all of $M$.
\end{proof}

\begin{theorem}\label{thm : star is clean}
 Let $L$ and $R$ be Dirac structures on $M$ whose tangent product $L \star R$ is again a Dirac structure. 
 \begin{enumerate}
  \item The Dirac structure $R$ induces a Dirac structure on each leaf of $L$, and the Dirac structure $L$ induces a Dirac structure on each leaf of $R$;
  \item Leaves of $L$ meet leaves of $R$ cleanly, and the leaf $\mathrm{S}_{L \star R}(x)$ of $L \star R$ through $x \in M$ is the connected component of $\mathrm{S}_{L}(x) \cap \mathrm{S}_{R}(x)$ through $x$.
 \end{enumerate}
\end{theorem}
\begin{proof}
We split the proof into four steps. In the first, we prove item 1; in the second and third, we describe the leaves of $L \star R$, and in the fourth we use Lemma \ref{lem : clean iff} to show that leaves meet cleanly. \\
\noindent \emph{Step 1. Leaves have induced Dirac structures.} Denote by
\begin{align*}
 i_{\mathrm{S}_L(x)} :\mathrm{S}_L(x) \to M
\end{align*}
the inclusion of the leaf of $L$ through $x \in M$. Because $N^*\mathrm{S}_L(x) \subset L \star R$, we may choose a splitting of vector bundles
\begin{align*}
 (L \star R)|_{\mathrm{S}_L(x)} =  N^*\mathrm{S}_L(x) \oplus C.
\end{align*}
This implies that the smooth vector bundle map
\begin{align*}
 & C \to \mathbb{T}\mathrm{S}_L(x), & u+\xi \mapsto u+{i_{\mathrm{S}_L(x)}}^*(\xi)
\end{align*}
has image ${i_{\mathrm{S}_L(x)}}^!(L \star R)$ and is a vector bundle isomorphism by a dimension count. In particular, ${i_{\mathrm{S}_L(x)}}^!(L \star R)$ is smooth and hence Dirac. In fact, there is an induced vector bundle map
\begin{align*}
 \overline{i_{\mathrm{S}_L(x)}} : {i_{\mathrm{S}_L(x)}}^!(L \star R) \to L \star R
\end{align*}
given by the composition
\begin{align*}
 {i_{\mathrm{S}_L(x)}}^!(L \star R) \stackrel{\simeq}{\longleftarrow} C \longrightarrow {i_{\mathrm{S}_L(x)}}^*(L \star R) \longrightarrow L \star R
\end{align*}where the second arrow is induced by the inclusion $C \subset L \star R$ and the third is the pullback map of vector bundles. Note finally that ${i_{\mathrm{S}_L(x)}}^!(L) = \mathrm{Gr}(\omega_{\mathrm{S}_L(x)})$, where $\omega_{\mathrm{S}_L(x)}$ denotes the closed two-form induced on $\mathrm{S}_L(x)$ by $L$. Therefore
\begin{align*}
 {i_{\mathrm{S}_L(x)}}^!(L \star R) = \mathrm{Gr}(\omega_{\mathrm{S}_L(x)}) \star {i_{\mathrm{S}_L(x)}}^!(R)
\end{align*}
implies that ${i_{\mathrm{S}_L(x)}}^!(R)$ is a Dirac structure as well. Arguing symmetrically with the roles of $L$ and $R$ reversed, we deduce that also the pullbacks of $L$ and $L \star R$ under the inclusion
\begin{align*}
 i_{\mathrm{S}_R(x)} :\mathrm{S}_R(x) \to M
\end{align*}
of the leaf of $R$ through $x$ is Dirac. This proves 1.\\

\noindent \emph{Step 2. The three induced partitions coincide.} As we observed in the previous step, for every $x \in M$, 
\begin{align*}
 {i_{\mathrm{S}_L(x)}}^!(L \star R) \subset \mathbb{T} \mathrm{S}_L(x)
\end{align*}
is a well-defined Dirac structure on $\mathrm{S}_L(x)$. The leaf through $x$ of the Dirac structure ${i_{\mathrm{S}_L(x)}}^!(L \star R)$ is a subset of the leaf $\mathrm{S}_{L \star R}(x)$ of $L \star R$ through $x$:
\begin{align*}
 \mathrm{S}_{{i_{\mathrm{S}_L(x)}}^!(L \star R)}(x) \subset \mathrm{S}_{L \star R}(x).
\end{align*}
This is because the composition of an ${i_{\mathrm{S}_L(x)}}^!(L \star R)$-path
\begin{align*}
 a:I \to  {i_{\mathrm{S}_L(x)}}^!(L \star R)
\end{align*}
with the bundle map
\begin{align*}
 \overline{i_{\mathrm{S}_L(x)}} : {i_{\mathrm{S}_L(x)}}^!(L \star R) \to L \star R
\end{align*}
is an $L \star R$-path, since $\overline{i_{\mathrm{S}_L(x)}}$ preserves anchors. We therefore have on $M$ two induced partitions by leafy submanifolds:
\begin{align*}
 && \mathscr{P}_1(x)=\mathrm{S}_{L \star R}(x), 
 && \mathscr{P}_2(x)=\mathrm{S}_{{i_{\mathrm{S}_L(x)}}^!(L \star R)}(x).
\end{align*}
The partition $\mathscr{P}_2$ refines $\mathscr{P}_1$, and because
\begin{align*}
 && T_x\mathscr{P}_1(x) = T_x\mathscr{P}_2(x) && \Longrightarrow && \dim \mathscr{P}_1(x) = \dim \mathscr{P}_2(x),
\end{align*}
we deduce that
\begin{align*}
 \mathscr{P}_1(x) = \coprod_{y \in \Upsilon_{12}}\mathscr{P}_2(y)
\end{align*}
is a disjoint union of open submanifolds, where $\Upsilon_{12} \subset \mathscr{P}_1(x)$. Because $\mathscr{P}_1(x)$ is connected, it then follows that
\begin{align*}
 \mathscr{P}_1(x) = \mathscr{P}_2(x).
\end{align*}
Symmetrically, the partitions $\mathscr{P}_1$ and $\mathscr{P}_2$ also agree with that given by pullbacks of $L \star R$ to leaves of $R$:
\begin{align*}
 \mathscr{P}_3(x):=\mathrm{S}_{{i_{\mathrm{S}_R(x)}}^!(L \star R)}(x).
\end{align*}

\noindent \emph{Step 3. Smooth path-connected components and $\star$.} Two points $y_0,y_1 \in Y$ of a subset $Y \subset M$ of a smooth manifold lie in the same \emph{smooth connected component} if there is a smooth curve
\begin{align*}
 && c: I \to M, && \text{such that} && c(0)=y_0, && c(I) \subset Y, && c(1)=y_1.
\end{align*}
This is an equivalence relation on $Y$, and we denote by $\langle \langle Y \rangle \rangle^M_y$ the smooth path-connected component containing $y$. By Step 2 we know that
\begin{align*}
 \mathrm{S}_{L \star R}(x) = \mathrm{S}_{{i_{\mathrm{S}_L(x)}}^!(L \star R)}(x) = \mathrm{S}_{{i_{\mathrm{S}_R(x)}}^!(L \star R)}(x),
\end{align*}
which in particular implies that
\begin{align*}
 \mathrm{S}_{L \star R}(x) \subset \langle \langle \mathrm{S}_L(x) \cap \mathrm{S}_R(x) \rangle \rangle^M_x.
\end{align*}
Conversely, any smooth path $c:I \to \mathrm{S}_L(x) \cap \mathrm{S}_R(x)$ starting at $c(0)=x$ lifts to $L$- and $R$-paths
\begin{align*}
 & a:TI \to L, & b: TI \to R,
\end{align*}covering $c$, in which case
\begin{align*}
 & a \star b:TI \to L \star R, & (a \star b)(t)=a(t)+\mathrm{pr}_{T^*}(b(t))
\end{align*}is an $L \star R$-path covering $c$. Therefore,
\begin{align*}
 \mathrm{S}_{L \star R}(x) = \langle \langle \mathrm{S}_L(x) \cap \mathrm{S}_R(x) \rangle\rangle^M_x.
\end{align*}

\noindent \emph{Step 4. Leaves of $L$ meet leaves of $R$ cleanly.} According to Step 3, we have a disjoint union
\begin{align*}
 \mathrm{S}_L(x) \cap \mathrm{S}_R(x) = \coprod_{y \in \Upsilon_{LR}}\mathrm{S}_{L \star R}(y),
\end{align*}
where $\Upsilon_{LR} \subset \mathrm{S}_L(x) \cap \mathrm{S}_R(x)$. Because
\begin{align*}
 T_y\mathrm{S}_{L \star R}(x) = T_y\mathrm{S}_{L \star R}(y) = T_y\mathrm{S}_L(y) \cap T_y\mathrm{S}_R(y) = T_y\mathrm{S}_L(x) \cap T_y\mathrm{S}_R(x)
\end{align*}
for all $y \in \mathrm{S}_{L \star R}(x)$, the conditions of item 3 of Lemma \ref{lem : clean iff} are fulfilled. Therefore $\mathrm{S}_L(x)$ and $\mathrm{S}_R(x)$ meet cleanly for every $x \in M$, and
\begin{align*}
 \mathrm{S}_{L \star R}(x) = \langle \langle \mathrm{S}_L(x) \cap \mathrm{S}_R(x) \rangle\rangle^M_x
\end{align*}
is the connected component of $\mathrm{S}_L(x) \cap \mathrm{S}_R(x)$ thorough $x$. This completes the proof of 2.
\end{proof}

\subsection*{Two illustrative examples}

A few comments are in order to appreciate the message conveyed by the (admittedly technical) Theorem \ref{thm : star is clean}. One of the messages of the theorem is that the only instance in which $L \star R$ may be smooth is if leaves of $L$ and leaves of $R$ meet cleanly. If leaves of $L$ and of $R$ meet \emph{transversally},
\begin{align*}
 & T_x\mathrm{S}_L(x) + T_x\mathrm{S}_R(x) = T_xM, & x \in \mathrm{S}_L(x) \cap \mathrm{S}_R(x),
\end{align*}
then a straightforward application of Criterion \ref{star criterion} shows that $L \star R$ is smooth. If leaves of $L$ are $R$ merely meet cleanly, one cannot employ Criterion \ref{star criterion} to deduce that $L \star R$ is smooth, but one may still wonder whether that condition is sufficient for smoothness of $L \star R$. Our first example shows that that is not the case: $L \star R$ may fail to be smooth even when leaves meet cleanly:

\begin{example}\normalfont
 Consider on $M=\mathbb{R}^2$ the Dirac structures
 \begin{align*}
  & L = \langle x_1\tfrac{\partial}{\partial x_1}-\mathrm{d} x_2,x_1\tfrac{\partial}{\partial x_2}+\mathrm{d} x_1 \rangle, & R = \langle \mathrm{d} x_1, \tfrac{\partial}{\partial x_2} \rangle.
 \end{align*}
 Observe that
 \begin{align*}
  & \mathrm{S}_L(x) \subset \mathrm{S}_R(x), & \text{if} \ x_1 = 0,
 \end{align*}
 while
 \begin{align*}
  & \mathrm{S}_L(x) \supset \mathrm{S}_R(x), & \text{if} \ x_1 \neq 0.
 \end{align*}
 This implies that leaves of $L$ and of $R$ meet cleanly. Moreover, by dimensional reasons, leaves of $L$ have induced $R$ structures, and leaves of $R$ have induced $L$ structures. However,
 \begin{align*}
  (L \star R)_x = \begin{cases}
                T^*_xM & \text{if} \ x_1 = 0,\\
               \langle \mathrm{d} x_1, \tfrac{\partial}{\partial x_2} \rangle, & \text{if} \ x_1 \neq 0
              \end{cases}
 \end{align*}
 is not smooth.
\end{example}

Our next example is somewhat more subtle: it is an example of two Dirac structures $L$ and $R$ in which leaves of $R$ all have Dirac structures induced by $L$, but not the other way around. That suffices to conclude from Theorem \ref{thm : star is clean} that $L \star R$ is not smooth, but we note yet another reason for the lack of smoothness: that the jumping phenomenon which Theorem \ref{thm : star is clean} is supposed to rule out, occurs on the Dirac structures induced on certain leaves of $R$. Another curious fact about the example to follow is that, while $L \star R$ is not smooth, it agrees on an open, dense set with a smooth Dirac structure. 

\begin{example}\normalfont
 Consider on $M=\mathbb{R}^4$ the Dirac structure
 \begin{align*}
  L = \langle \tfrac{\partial}{\partial x_1} - \mathrm{d} x_2,  \tfrac{\partial}{\partial x_2} + \mathrm{d} x_1, x_3\tfrac{\partial}{\partial x_3} - \mathrm{d} x_4, x_3\tfrac{\partial}{\partial x_4}+\mathrm{d} x_3 \rangle
 \end{align*}
 which corresponds to the Poisson structure
 \begin{align*}
  \pi = \tfrac{\partial}{\partial x_1} \wedge \tfrac{\partial}{\partial x_2} + x_3\tfrac{\partial}{\partial x_3} \wedge \tfrac{\partial}{\partial x_4}.
 \end{align*}
 Let also $f$ denote the smooth function 
 \begin{align*}
  & f : M \to \mathbb{R}, & f(x_1,x_2,x_3,x_4)=\tfrac{1}{2}(x_1^2+x_2^2),
 \end{align*}
 and consider the surjective submersion
 \begin{align*}
  & p:M \to \mathbb{R}^2, & p(x_1,x_2,x_3,x_4) = (x_3-f,x_4-f).
 \end{align*}
 The Dirac structure corresponding to the foliation by fibres of $p$ is
 \begin{align*}
  & R = \left\langle v_1, v_2 ,\mathrm{d} x_3-\mathrm{d} f, \mathrm{d} x_4-\mathrm{d} f \right\rangle, & \text{where} \ \ v_i = \tfrac{\partial}{\partial x_i}+x_i\left(\tfrac{\partial}{\partial x_3}+\tfrac{\partial}{\partial x_4}  \right).
 \end{align*}
 The tangent product of $L$ and $R$ is given by
 \begin{align*}
  L \star R = \begin{cases}
               \left\langle x_2v_1 - x_1v_2, \mathrm{d} f, \mathrm{d} x_3, \mathrm{d} x_4\right\rangle & \text{on} \ A,\\
               \left\langle v_1 - \mathrm{d} x_2, v_2+\mathrm{d} x_1,\mathrm{d} x_3-\mathrm{d} f, \mathrm{d} x_4-\mathrm{d} f \right\rangle & \text{on} \ M \diagdown A,               
              \end{cases}
 \end{align*}
 where $A \subset M$ denotes the subset
 \begin{align*}
   A = (\mathbb{R}^2 \diagdown \{(0,0)\}) \times \{0\} \times \mathbb{R}.
 \end{align*} 
Note that:
\begin{enumerate}[a)]
 \item The parity of $L \star R$ on points of $A$ and of $M \diagdown A$ is different; in particular, $L \star R$ is not smooth; 
 \item The open, dense set
 \begin{align*}
  U:=M \diagdown \mathrm{S}_L(0) = \{ x \in M \ | \ x_3 \neq 0\}
 \end{align*}
 is the maximal open set to which Criterion \ref{star criterion} applies to $L$ and $R$, and the restriction of $L \star R$ to $U$ has an extension to a Dirac structure on all of $M$ --- namely, the graph of the rank-two Poisson structure
 \begin{align*}
  & \Pi \in \Gamma(\wedge^2T\mathscr{F}), & \Pi = v_1 \wedge v_2.
 \end{align*}

 \item Leaves of $R$ have induced $L$-structures. Indeed, each leaf $i_R:\mathrm{S}_R(x) \to M$ of $R$ meets the section $\{(0,0)\} \times \mathbb{R}^2$ of $p$ the at a single point $(0,0,x_3,x_4)$, and coincides with the image of the embedding
 \begin{align*}
  & j:\mathbb{R}^2 \to M, & j(x_1,x_2)=(x_1,x_2,x_3+f,x_4+f).
 \end{align*}
 Hence each leaf of $R$ inherits a symplectic structure $i_R^!(L)$, since
 \begin{align*}
  j^!(L) = \mathrm{Gr}(\tfrac{\partial}{\partial x_1} \wedge \tfrac{\partial}{\partial x_2}).
 \end{align*}
Note however that leaves
\begin{align*}
 & \mathrm{S}_R(0,0,x_3,x_4), & x_3 \leqslant 0
\end{align*}meet the singular locus
\begin{align*}
 & \mathrm{Sing}(L) = \{ x \in M \ | \ x_3=0\}
\end{align*}
and therefore exhibit the jumping phenomenon. For example, the curve
\begin{align*}
 & c: I \to \mathrm{S}_R(0), & c(t) = (t,0,\tfrac{1}{2}t^2,\tfrac{1}{2}t^2)
\end{align*}lifts to an $i_R^!(L)$-path
\begin{align*}
 & a_{i_R^!(L)}: TI \to i_R^!(L), & a_{i_R^!(L)} = -\mathrm{d} {x_2}_{c(t)},
\end{align*}but the unique $L$-path $ a_{L}: TI \to L$ which projects to $c$ is
\begin{align*}
 a_{L} = \begin{cases}
          \left(-\mathrm{d} x_2 + \tfrac{x_1}{x_3}(\mathrm{d} x_3-\mathrm{d} x_4)\right)_{c(t)}, & \text{if} \ t \neq 0;\\
          -\mathrm{d} {x_2}_{c(0)}, & \text{if} \ t = 0,
         \end{cases}
\end{align*}
which is \emph{discontinuous} at $t=0$. This reflects the fact that $i_R^!(L)$-paths need not be $i_R$-related to $L$-paths; said otherwise, that the leaf $\mathrm{S}_{i_R^!(L)}(0)$ of $i_R^!(L)$ on $\mathrm{S}_{R}(0)$ is not the connected component of the intersection of $\mathrm{S}_{L}(0) \cap \mathrm{S}_{R}(0)$ which contains the origin, since
\begin{align*}
 & \mathrm{S}_{i_R^!(L)}(0)=\mathrm{S}_{R}(0), & \mathrm{S}_{L}(0) \cap \mathrm{S}_{R}(0) = \{0\}
\end{align*}

\item Not all leaves $i_L: \mathrm{S}_L(x) \to M$ of $L$ inherit $R$-structures: for instance, on the leaf through the origin,
\begin{align*}
 \mathrm{S}_L(0) = \mathbb{R}^2 \times \{(0,0)\}
\end{align*}the induced Lagrangian family is
\begin{align*}
 i_L^!(R) = \begin{cases}
             \langle \pi^{\sharp}(\mathrm{d} f),\mathrm{d} f\rangle & \text{outside the origin}\\
             T_0\mathrm{S}_L(0)& \text{at the origin}
            \end{cases}
\end{align*}
and is therefore not smooth.
\end{enumerate}
\end{example}

\section{Cotangent products}\label{sec : Cotangent product}

As we discussed so far, the set of Lagrangian families has a well-defined product $\star$, which has the following simple interpretation: if two Dirac structures $L$ and $R$ have a smooth tangent product $L \star R$, then $L \star R$ is Dirac, with leaves given by the connected components of the intersections of leaves of $L$ and $R$, and closed two-form the sum of those coming from $L$ and from $R$.

The purely algebraic description of $\star$ --- as opposed to its definition as a pullback Lagrangian family --- can be dualized into a product:
\begin{align}\label{eq : circledast}
& \circledast: \mathrm{Lag}(\mathbb{T} M) \times \mathrm{Lag}(\mathbb{T} M) \longrightarrow \mathrm{Lag}(\mathbb{T} M), & (L,R) \mapsto L \circledast R,
\end{align}explicitly given by:
\begin{align}\label{eq : explicit definition of circledast}
 L \circledast R = \{ a + \mathrm{pr}_T(b) = \mathrm{pr}_T(a)+b \ | \ (a,b) \in L \times R, \ \mathrm{pr}_{T^*}(a-b)=0\}.
\end{align}

\begin{lemma}[Cotangent product]\label{lem : cotangent product}
 The map $\circledast$ defines an associative and commutative product on Lagrangian families, for which $T^*M$ acts as unity and $TM$ as zero:
  \begin{multicols}{2}
 \begin{enumerate}[a)]
 \setcounter{enumi}{\value{conditions2}}
 \item $L \circledast (R \circledast S) = (L \circledast R) \circledast S$;
 \item $L \circledast R=R \circledast L$;
 \setcounter{conditions2}{\value{enumi}}
\end{enumerate}
\begin{enumerate}[a)]
 \setcounter{enumi}{\value{conditions2}}
 \item $L \circledast T^*M = L$;
 \item $L \circledast TM=TM$.
 \setcounter{conditions2}{\value{enumi}}
\end{enumerate}
 \end{multicols}Moreover, it is compatible with rescaling in the sense that $\mathcal{R}_t(L \circledast R) = \mathcal{R}_t(L) \circledast \mathcal{R}_t(R)$.
\end{lemma}
\begin{proof}
 That the cotangent product $L \circledast R$ of two Lagrangian families $L,R$ is again Lagrangian can be argued as follows: for any Lagrangian family $L$ on $M$ and $x \in M$, there is vector space $F \subset T_xM$ and a bivector $\pi \in \wedge^2T_xM$, with the property that $L = \mathrm{Gr}(\pi) \circledast \mathrm{Gr}(F)$. Indeed, let $F:=\mathrm{pr}_{T^*}(L)^{\circ}$ and any section $\nu : F^{\circ} \to L$. Then $\langle \nu(\xi),\eta\rangle$ is skew-symmetric in $\xi,\eta \in F^{\circ}$ because $L$ is isotropic, and so can be extended to a bivector $\pi \in \wedge^2T_xM$. This means that $\mathrm{Gr}(\pi) \circledast \mathrm{Gr}(F) \subset L$, and so equality ensues. The remainder of the proof is analogous to that of Lemma \ref{lem : tangent product}.
\end{proof}

As expected, smoothness is not preserved by $\circledast$:

\begin{example}\normalfont
If $L$ is a Lagrangian family, then
\begin{align*}
 & L \circledast \mathcal{R}_{-1}(L) = \mathrm{pr}_T(L) \oplus \mathrm{pr}_T(L)^{\circ}.
\end{align*}Hence for a smooth $L$, $L \circledast \mathcal{R}_{-1}(L)$ is smooth exactly when $\mathrm{pr}_T(L) \subset TM$ has locally constant rank.
\end{example}

The criterion dual to Criterion \ref{star criterion} is
\begin{criterion}\label{circledast criterion}
 For Lagrangian subbundles $L,R \subset \mathbb{T}M$, a sufficient condition for $L \circledast R$ to be smooth is that the vector bundle map
 \begin{align*}
  & \delta_{T^*} : L \oplus R \to T^*M, & \delta_{T^*}(a,b):=\mathrm{pr}_{T^*}(a-b)
 \end{align*}have locally constant rank. If that is the case, then
 \begin{align*}
  \Gamma(L \circledast R) = \{a+\mathrm{pr}_{T}(b) \ | \ a \in \Gamma(L), \ b \in \Gamma(R), \ \mathrm{pr}_{T^*}(a)=\mathrm{pr}_{T^*}(b)\}.
 \end{align*}
\end{criterion}

The most transparent examples of cotangent products are those between bivectors and between vector subbundles of $TM$, in which case the cotangent product reduces to addition:

\begin{example}\label{example : coproduct as gauge}\normalfont
 The cotangent product of the graph $\mathrm{Gr}(\pi)$ of a bivector $\pi$ with any Lagrangian subbundle $L$ is smooth, and 
 \[\mathrm{Gr}(\pi) \circledast L = \mathcal{R}_{\pi}L = \{a+\pi^{\sharp}\mathrm{pr}_{T^*}(a)\ | \ a \in L\}.\]In particular, $\mathrm{Gr}(\pi) \circledast \mathrm{Gr}(\pi')=\mathrm{Gr}(\pi+\pi')$. 
\end{example}

\begin{example}\label{example : foliations 2}\normalfont
 If $E,F \subset TM$ are vector subbundles, and $L=\mathrm{Gr}(E)$, $R=\mathrm{Gr}(F)$ are the ensuing Lagrangian subbundles, then
 \begin{align*}
  L \circledast R & =(E + F)\oplus (E^{\circ} \cap F^{\circ})
 \end{align*}is smooth exactly when $E + F \subset TM$ is a vector subbundle.
\end{example}

\begin{lemma}
If two Lagrangian families $L$ and $R$ on $M$ are pushed forward by a smooth map $\phi : M \to N$, then so does $L \circledast R$, and
  \begin{align*}
   \phi_!(L \circledast R) = \phi_!(L) \circledast \phi_!(R).
  \end{align*}  
\end{lemma}

\begin{remark}\normalfont
 One of the algebraic difficulties of tangent- and cotangent products is the lack of a distributive property between the two products $\star$ and $\circledast$. For example, none of the relations below
 \begin{align*}
  && L_1 \star (L_2 \circledast L_3) = (L_1 \star L_2) \circledast L_3, && L_1 \star (L_2 \circledast L_3) = (L_1 \star L_2) \circledast (L_1 \star L_3)
 \end{align*}
holds true in general. For example, they do not hold if $L_1=L_2=L_3$ is the graph of the closed two-form $\mathrm{d} x \wedge \mathrm{d} y$ on $M=\mathbb{R}^2$.
\end{remark}

\section{Concurrence}

While the discussions about tangent- and cotangent products ran essentially in parallel so far, the seeming symmetry between $\star$ and $\circledast$ breaks down when it comes to the involutivity condition
\[
 [\Gamma(L),\Gamma(L)] \subset \Gamma(L)
\]which selects Dirac structures among Lagrangian subbundles. While the tangent product of Dirac structures which is again smooth is automaticaly Dirac, the analogous statement for cotangent products is false. This reflects the very different behavior the Dorfman bracket displays when restricted to the tangent or the cotangent bundles, as well as the absence of a geometric description for $L \circledast R$ as we had for $L \star R = \varDelta^!(L \times R)$. This motivates the central definition of the paper:\\

\begin{maindefinition}\label{def : concurring Dirac}
 Two Dirac structures $L$ and $R$ {\bf concur} if $L \circledast R$ is a Dirac structure. 
\end{maindefinition}

\begin{remark}\normalfont
As pointed out earlier, in Dirac geometry it is convenient to separate --- as much as possible -- algebraic conditions from differential geometric ones: we call an arbitrary Lagrangian family {\bf involutive} if, for any three \underline{smooth} local sections $a,b,c \in \Gamma(U,L)$, the expression $\langle [a,b],c\rangle \in C^{\infty}(U)$ vanishes identically. For a Lagrangian subbundle, the expression at hand is the evaluation of the Courant tensor of $L$ on these three sections; hence Dirac structures are those Lagrangian families which are smooth and involutive. The condition of involutivity may be rather meaningless --- for instance for Lagrangian families without non-trivial smooth sections  --- and we shall only find it useful for Lagrangian families which are smooth on an open, dense subset of $M$
(c.f. Remark \ref{rem : iterated star criterion}).
\end{remark}

\begin{definition}\label{def : weak concurring}
 Two Dirac structures $L$ and $R$ {\bf concur weakly} if the Lagrangian family $L \circledast R$ is involutive, in which case we write $L \smile R$.
\end{definition}

\begin{lemma}
 Two Dirac structures concur weakly exactly when they concur on an open, dense subset.
\end{lemma}
\begin{proof}
 Let $L$ and $R$ be Dirac structures on $M$ which concur weakly. Then the largest open subset $U$ of $M$ for which Criterion \ref{circledast criterion} applies to $L|_U$ and $R|_U$ is dense. Therefore $L|_U \circledast R|_U$ is smooth, whence $L|_U $ and $R|_U$ concur. Conversely, suppose $U \subset M$ is an open, dense subset on which $L|_U$ and $R|_U$ concur. Suppose $a,b,c \in \Gamma(V,L\circledast R)$ are local sections, for some open set $V \subset M$. Because $V \cap U$ is dense in $V$, and the restriction of $\langle [a,b],c\rangle \in C^{\infty}(V)$ to $V \cap U$ vanishes, we conclude that $\langle [a,b],c\rangle$ vanishes identically. Hence $L$ and $R$ concur weakly.
\end{proof}

In the two examples below, we mention the two most transparent --- even prototypical --- manifestations of the concurrence condition.

\begin{example}\label{example : coproduct as gauge inv}\normalfont
 The graphs $L=\mathrm{Gr}(\pi_L)$ and $R=\mathrm{Gr}(\pi_R)$ of Poisson structures concur iff $\pi_L$ and $\pi_R$ commute as Poisson structures.
\end{example}

\begin{example}\label{example : foliations 3}\normalfont
 The graphs $L=\mathrm{Gr}(\mathscr{F}_L)$ and $R=\mathrm{Gr}(\mathscr{F}_R)$ of foliations on $M$ concur iff $T\mathscr{F}_L+T\mathscr{F}_R$ is the tangent bundle to a foliation on $M$.
\end{example}

The condition that Dirac structures concur weakly is rather delicate. For example, if Dirac structures $L_1,L_2,L_3$ are such that $L_1 \smile L_2$ and $L_2 \smile L_3$, it does not follow that $L_1 \smile L_3$ (for example, the first condition is always satisfied when $L_2=TM$, regardless of who $L_1$ and $L_3$ are). We have, however, the following quadratic relations between Courants tensors of cotangent products:

\begin{proposition}\label{pro : concurrence pairwise and scalars}
 Let $L_1,...,L_n \subset \mathbb{T}M$ be Lagrangian subbundles. 
 \begin{enumerate}[a)]
  \item The cotangent product
 \begin{align*}
  L_{1} \circledast L_{2} \circledast \cdots \circledast L_{n}
 \end{align*}is involutive if each Lagrangian subbundle is a Dirac structure, and any two Dirac structures concur weakly:
 \begin{align*}
  & L_{i} \smile L_{j}, & 1 \leqslant i,j \leqslant n.
 \end{align*}
 \item For all $t_1,...,t_n \in \mathbb{R}^{\times}$, we have that
 \begin{align*}
  && \mathcal{R}_{t_1}(L_1) \circledast \cdots \circledast \mathcal{R}_{t_n}(L_n) \ \ \text{involutive} && \Longleftrightarrow && L_{1} \circledast L_{2} \circledast \cdots \circledast L_{n}\ \ \text{involutive.}
 \end{align*}

 \end{enumerate}
\end{proposition}
\begin{proof}
 There exists an open, dense subset $U \subset M$, on which the cotangent product
 \begin{align*}
    R:=(L_{1} \circledast L_{2} \circledast \cdots \circledast L_{n})|_U
 \end{align*}is smooth, and sections $a \in \Gamma(R)$ are of the form
 \begin{align}\label{eq : form of sections}
  && a = \sum_{i=1}^n u_{i} + \xi, && \text{where} \ \  u_{1},...,u_{n} \in \mathfrak{X}(U) \ \ \text{and} \ \  \xi \in \Omega^1(U).
 \end{align}
 For such a section, we write $a_{i} \in \Gamma(L_{i}|_U)$ and $a_{ij} \in \Gamma\left(L_{i}|_U \circledast L_j|_U\right)$ for the sections
 \begin{align*}
  & a_{i}:=u_{i}+\xi, & a_{ij}:=u_{i}+u_{j}+\xi.
 \end{align*}
 \noindent \emph{Proof of 1.} Consider sections
 \begin{align*}
  && a = \sum_{i=1}^n u_{i} + \xi, && b = \sum_{i=1}^n v_{i} + \eta, && c = \sum_{i=1}^n w_{i} + \zeta
 \end{align*}of $R$ of the form \eqref{eq : form of sections}. Then
 \begin{align*}
  [a,b] = \sum_{i=1}^n \left[a_{i},b_{i}\right]+\sum_{i \neq j} \left[u_{i},v_{j}\right]
 \end{align*}
and so
 \begin{align*}
  \Upsilon_{R}(a,b,c) 
        & = \langle [a,b],c\rangle = \sum_i \langle \left[a_{i},b_{i}\right],c\rangle +\sum_{i \neq j} \langle \left[u_{i},v_{j}\right],c \rangle\\
        & = \sum_{i} \left\langle  \left[a_{i},b_{i}\right],c_{i} \right\rangle + \sum_{i \neq j}\left\langle  \left[a_{i},b_{i}\right],w_{j} \right\rangle + \sum_{i \neq j}\left\langle  \left[u_{i},v_{j}\right],\zeta \right\rangle\\
        & = \sum_{i}\Upsilon_{L_{i}}(a_{i},b_{i},c_{i}) + \sum_{i < j}\left( \langle \left[a_{i},b_{i}\right],w_{j} \rangle + \langle [u_i,v_j]+[u_j,v_i],\zeta \rangle + \langle \left[a_{j},b_{j}\right],w_{i} \rangle\right)\\
        & = (2-n)\sum_{i}\Upsilon_{L_{i}}(a_{i},b_{i},c_{i}) + \sum_{i < j}\left(\left\langle  \left[u_i+u_j+\xi,v_i+v_j+\eta\right],w_i+w_j+\zeta \right\rangle \right)\\
        & = (2-n)\sum_{i}\Upsilon_{L_{i}}(a_{i},b_{i},c_{i}) + \sum_{i<j}\Upsilon_{L_{i} \circledast L_{j}}(a_{ij},b_{ij},c_{ij}).
 \end{align*}
 Therefore
 \begin{align*}
  \Upsilon_{R}(a,b,c)  = (2-n)\sum_{i}\Upsilon_{L_{i}}(a_i,b_i,c_i) + \sum_{i<j}\Upsilon_{L_{i} \circledast L_{j}}(a_{ij},b_{ij},c_{ij})
 \end{align*}ensures that $R$ is Dirac on $U$ if all Lagrangian subbundles $L_i$ are Dirac, and any two such Dirac structures concur weakly.
 
 \noindent \emph{Proof of 2.} Let $t_1,...,t_n$ be non-zero real numbers. Then
 \begin{align*}
  && a' = \sum_{i=1}^n t_iu_i + \xi, && b' = \sum_{i=1}^n t_iv_i + \eta, && c' = \sum_{i=1}^n t_iw_i + \zeta
 \end{align*}are sections of $\mathcal{R}_{t_1}(L_1) \circledast \cdots \circledast \mathcal{R}_{t_n}(L_n)$ over $U$. Because
 \begin{align*}
  \Upsilon_{\mathcal{R}_{t_i}(L_i)}(a'_i,b'_i,c'_i) & = \left\langle \left[t_iu_i + \xi,t_iv_i + \eta \right],t_iw_i + \zeta \right\rangle\\
                                                                & = \left\langle \left[t_iu_i,t_iv_i \right],\zeta \right\rangle + \left\langle \left[t_iu_i,\eta \right]+[\xi,t_iv_i],t_iw_i\right\rangle\\
                                                                & = t_i^2\Upsilon_{L_{i}}(a_i,b_i,c_i)
 \end{align*}
and
 \begin{align*}
  \Upsilon_{\mathcal{R}_{t_i}(L_i) \circledast \mathcal{R}_{t_j}(L_j)}(a'_{ij},b'_{ij},c'_{ij}) & = \left\langle \left[t_iu_i+t_ju_j + \xi,t_iv_i+t_jv_j + \eta \right],t_iw_i +t_jw_j + \zeta \right\rangle\\
                                                                & = \left\langle \left[t_iu_i + \xi,t_iv_i + \eta \right],t_jw_j\right\rangle \\
                                                                & + \left\langle [t_iu_i,t_jv_j] + [t_ju_j,t_iv_i],\zeta\right\rangle\\
                                                                & + \left\langle \left[t_ju_j + \xi,t_jv_j + \eta \right],t_iw_i\right\rangle\\
                                                                & -\Upsilon_{L_i}(a_i',b_i',c_i') -\Upsilon_{L_j}(a_j',b_j',c_j')\\
                                                                & = t_it_j\left\langle \left[u_i+u_j + \xi,v_i+v_j + \eta \right],w_i +w_j + \zeta \right\rangle\\
                                                                & -t_i^2\Upsilon_{L_i}(a_i,b_i,c_i) -t_j^2\Upsilon_{L_j}(a_j,b_j,c_j)\\
                                                                & = t_it_j\Upsilon_{L_{i} \circledast L_{j}}(a_{ij},b_{ij},c_{ij})-t_i^2\Upsilon_{L_i}(a_i,b_i,c_i) -t_j^2\Upsilon_{L_j}(a_j,b_j,c_j)
 \end{align*}
 we deduce by the proof of item 1 that
 \begin{align*}
  \Upsilon_{\mathcal{R}_{t_1}(L_1) \circledast \cdots \circledast \mathcal{R}_{t_n}(L_n)}(a',b',c') 
        & = (2-n)\sum_{i}\Upsilon_{\mathcal{R}_{t_i}(L_i)}(a'_i,b'_i,c'_i) + \sum_{i<j}\Upsilon_{\mathcal{R}_{t_i}(L_i) \circledast \mathcal{R}_{t_j}(L_j)}(a'_{ij},b'_{ij},c'_{ij})\\
        & = (3-2n)\sum_{i}t_i^2\Upsilon_{L_{i}}(a_i,b_i,c_i) + \sum_{i<j}t_it_j\Upsilon_{L_{i} \circledast L_{j}}(a_{ij},b_{ij},c_{ij}) .
 \end{align*}
 Therefore if each $L_i$ is Dirac and any two $L_i,L_j$ concur weakly, then $\mathcal{R}_{t_1}(L_1) \circledast \cdots \circledast \mathcal{R}_{t_n}(L_n)$ is involutive. Since $n$ is arbitrary, this in particular implies that each $\mathcal{R}_{t_i}(L_i)$ is Dirac, and that any two $\mathcal{R}_{t_i}(L_i),\mathcal{R}_{t_j}(L_j)$ concur weakly, which by the discussion above implies in its turn that
 \begin{align*}
  L_1 \circledast \cdots \circledast L_n = \mathcal{R}_{\tfrac{1}{t_1}}(\mathcal{R}_{t_1}(L_1)) \circledast \cdots \circledast \mathcal{R}_{\tfrac{1}{t_n}}(\mathcal{R}_{t_n}(L_n))
 \end{align*}
is involutive. 
\end{proof}

\section{Comparison with Dirac pairs}\label{sec : Comparison with Dirac pairs}

A notion closely related to our notion of concurrence had already been proposed in the literature: that of \emph{Dirac pairs} (see \cite[Section 6]{Dorfman}, \cite[Section 3.6]{Dorfman_book} and \cite[Definition 3.1]{Kosmann}). We briefly recall its setting here. Recall that a \emph{relation} between sets $X$ and $Y$ is a subset $L \subset X \times Y$. The \emph{inverse} relation is
\begin{align*}
 && L^{-1} \subset Y \times X, && L^{-1}:=\{(y,x) \ | \ (x,y) \in L\}.
\end{align*}The \emph{composition} of relations
\begin{align*}
 && L \subset X \times Y, && R \subset Y \times Z
\end{align*}is the relation
\begin{align*}
 && L \circ R \subset X \times Z, && L \circ R = \{(x,z) \ | \ \exists \ y \in Y \ \text{such that} \ (x,y) \in L, \ (y,z) \in R\}.
\end{align*}If $X$ and $Y$ are vector spaces, and $L \subset X \times Y$ is a relation, we refer to
\begin{align*}
 && L^{\vee} \subset X^* \times Y^*, && L^{\vee}:= \{(\xi,\eta) \ | \ (x,y) \in L \ \text{implies} \ \xi(x)=\eta(y)\}
\end{align*}as the relation \emph{dual} to $L$.

\begin{lemma}
 For all Lagrangian families $L,R \subset \mathbb{T} M$, define
 \begin{align*}
  & L \diamondsuit_T R:=L \circ R^{-1}, & L \diamondsuit_{T^*} R:=L^{-1} \circ R.
 \end{align*}
Then we have that
 \begin{align*}
  (L \diamondsuit_T R)^{\vee}=L \diamondsuit_{T^*} R.
 \end{align*}
\end{lemma}
\begin{proof}
First observe that, by definition,
\begin{align*}
 L \diamondsuit_T R & = \{(u,v) \in TM \oplus TM \ | \ \exists \ \xi \in T^*M, \ (u+\xi,v+\xi) \in L \oplus R\}\\
 L \diamondsuit_{T^*} R & = \{(\xi,\eta) \in T^*M \oplus T^*M \ | \ \exists \ w \in TM, \ (w+\xi,w+\eta) \in L \oplus R\}\\
 (L \diamondsuit_T R)^{\vee} & = \{(\xi,\eta) \in T^*M \oplus T^*M \ | \ (u,v) \in L \diamondsuit_T R \ \Rightarrow \ \langle \xi,u\rangle = \langle \eta,v\rangle\}.
\end{align*}
We show that the latter two coincide. Let $(\xi,\eta) \in L \diamondsuit_{T^*} R$ and let $(u,v) \in L \diamondsuit_{T} R$. Then there are $w \in TM$ and $\zeta \in T^*M$ for which
\begin{align*}
 && w+\xi \in L, && u+\zeta \in L, && w+\eta \in R, && v+\zeta \in R.
\end{align*}
Because $L$ and $R$ are Lagrangian, we deduce that
\begin{align*}
 \langle \xi,u\rangle = -\langle \zeta,w\rangle = \langle \eta,v\rangle.
\end{align*}This implies that
\begin{align*}
 L \diamondsuit_{T^*} R \subset (L \diamondsuit_{T} R)^{\vee}.
\end{align*}
To prove the converse, fix $x \in M$, and write $L_x,R_x$ as
\begin{align*}
 & L_x=\mathcal{R}_{\pi_L}(F_L\oplus F_L^{\circ}), & R_x=\mathcal{R}_{\pi_R}(F_R \oplus F_R^{\circ}),
\end{align*}where $F_L,F_R \subset T_xM$ are vector subspaces and $\pi_L,\pi_R \in \wedge^2T_xM$ are bivectors. 
Because
\begin{align*}
 (L \diamondsuit_{T} R)_x = F_L \times F_R + (\pi_L^{\sharp},\pi_R^{\sharp})(F_L^{\circ} \cap F_R^{\circ}),
\end{align*}
we see that $(L \diamondsuit_{T} R)^{\vee}_x$ consists of those pairs $(\xi,\eta)$, where
\begin{align*}
 && \xi \in F_L^{\circ}, && \eta \in F_R^{\circ}, && \pi_L^{\sharp}(\xi)-\pi_R^{\sharp}(\eta) \in F_L+F_R.
\end{align*}
Therefore $(\xi,\eta) \in F_L^{\circ} \times F_R^{\circ}$ lies in $(L \diamondsuit_{T} R)^{\vee}_x$ if and only if there exist $(w_L,w_R) \in F_L \times F_R$, such that
\begin{align*}
 w_L+\pi_L^{\sharp}(\xi) = w_R+\pi_R^{\sharp}(\eta).
\end{align*}Call this common vector $z$, and observe that
\begin{align*}
 & z + \xi \in L_x, & z + \eta \in R_x,
\end{align*}and therefore $z + \xi  + \eta \in (L \star R)_x$, which is to say that $(\xi,\eta) \in (L \diamondsuit_{T^*} R)_x$. Thus
\begin{align*}
 (L \diamondsuit_{T^*} R)_x \supset (L \diamondsuit_{T} R)^{\vee}_x,
\end{align*}and so equality follows.
\end{proof}

Given Lagrangian subbundles $L,R \subset \mathbb{T}M$, consider the space $\mathscr{P}(L,R) \subset \Gamma(\mathbb{T}M)^7$ consisting of all $7$-tuples of sections $(u_L,u_R,v_L,v_R,\zeta_R,\zeta_{LR},\zeta_L)$, where
\begin{align*}
 && (u_L,u_R), (v_L,v_R) \in L \diamondsuit_{T} R, && (\zeta_{LR},\zeta_R), (\zeta_L,\zeta_{LR}) \in L \diamondsuit_{T^*} R.
\end{align*}Note that $u_L,u_R,v_L,v_R \in \mathfrak{X}(M)$, while $\zeta_L,\zeta_{LR},\zeta_R \in \Omega^1(M)$.

\begin{definition}\label{def : Dirac pairs}
 The {\bf torsion} of a pair of Lagrangian subbundles $L,R \subset \mathbb{T}M$ is the function $\mathfrak{y}_{L,R} : \mathscr{P}(L,R) \to C^{\infty}(M)$ given by
\begin{align}\label{eq : torsion of a pair}
\mathfrak{y}_{L,R} \left(u_L,u_R,v_L,v_R,\zeta_L,\zeta_{LR},\zeta_R\right) = \langle [u_L,v_L], \zeta_L \rangle - \langle [u_L,v_R]+[u_R,v_L], \zeta_{LR}\rangle + \langle [u_R,v_R],\zeta_R \rangle.
\end{align}
Two Dirac structures $(L,R)$ form a {\bf Dirac pair} if their torsion $\mathfrak{y}_{L,R}$ vanishes identically.
\end{definition}

For example, two \emph{symplectic} Poisson structures $\pi_0$ and $\pi_1$, $\mathrm{Gr}(\pi_0)$ and $\mathrm{Gr}(\pi_1)$ form a Dirac pair (\cite[Theorem 3.2]{Kosmann} exactly when $\pi_0$ and $\pi_1$ commute. In fact, the notion of concurrence of Dirac structures is a refinement of the notion of Dirac pairs, in the following sense:

\begin{theorem}\label{thm : concurring Dirac structures form a Dirac pair}
 Weakly concurring Dirac structures form a Dirac pair.
\end{theorem}
\begin{proof}
Let $L$ and $R$ be Dirac structures on $M$ for which $L \circledast R$ is involutive. Choose a section
\begin{align*}
 (u_L,u_R,v_L,v_R,\zeta_R,\zeta_{LR},\zeta_L) \in \mathscr{P}(L,R).
\end{align*}Concretely, this means that there exist $w,z \in \mathfrak{X}(M)$ and $\xi,\eta \in \Omega^1(M)$, such that
\begin{align*}
  && a_{L}=u_L+\xi, && b_{L}=v_L+\eta, && w+\zeta_{LR}, && z+\zeta_{L} && \text{belong to} \ \Gamma(L)\\
  && a_{R}=u_R+\xi, && b_{R}=v_R+\eta, && w+\zeta_{R}, && z+\zeta_{LR} && \text{belong to} \ \Gamma(R).
\end{align*}which implies that all of
\begin{align*}
 a_{\circledast}:=u_L+u_R+\xi, \qquad\qquad b_{\circledast}:=v_L+v_R+\eta, \qquad\qquad c_{\circledast}:=w+z+\zeta_{LR}
\end{align*}lie in $L \circledast R$. Then observe that:
\begin{align*}
 [a_{\circledast},b_{\circledast}] = [a_{L},b_{L}] + [a_{R},b_{R}] + [u_L,v_R]+[u_R,v_L],
\end{align*}and therefore:
\begin{align*}
 0 & \stackrel{(1)}{=} \langle [a_{\circledast},b_{\circledast}],c_{\circledast}\rangle \\ & = \langle [a_{L},b_{L}] + [a_{R}, b_{R}] + [u_L,v_R]+[u_R,v_L],c_{\circledast}\rangle =\\
 & \stackrel{(2)}{=}\langle [a_{L},b_{L}], z \rangle + \langle [a_{R}, b_{R}], w \rangle +\langle [u_L,v_R]+[u_R,v_L],\zeta_{LR} \rangle = \\
 & \stackrel{(3)}{=}-\langle [u_L,v_L], \zeta_L \rangle - \langle [u_R,v_R],\zeta_R \rangle +\langle [u_L,v_R]+[u_R,v_L], \zeta_{LR} \rangle =\\
 & \stackrel{(4)}{=} -\mathfrak{y}_{L,R}\left(u_L,u_R,v_L,v_R,\zeta_R,\zeta_{LR},\zeta_L\right),
\end{align*}where (1) is by assumption that $L \circledast R$ is involutive, (2) is because
\begin{align*}
 & [a_{L},b_{L}], \quad z+\zeta_{LR} \ \in \ \Gamma(L), & [a_{R},b_{R}], \quad w+\zeta_{LR} \ \in \ \Gamma(R),
\end{align*}(3) is because
\begin{align*}
 & [a_{L},b_{L}], \quad w+\zeta_L \ \in \ \Gamma(L), & [a_{R},b_{R}], \quad w+\zeta_R \ \in \ R,
\end{align*}and (4) is by definition of torsion $\mathfrak{y}_{L,R}$. This shows that $(L,R)$ is a Dirac pair.
\end{proof}

Dirac pairs need not concur weakly, however:

\begin{example}\normalfont
\emph{An example of a Dirac pair $(L,R)$ which does not concur weakly.} On $M=\mathbb{R}^4$ consider the Poisson structures
\begin{align*}
 & \pi_L = \tfrac{\partial}{\partial x_1} \wedge \tfrac{\partial}{\partial x_2}, & \pi_R = x_1\tfrac{\partial}{\partial x_3} \wedge \tfrac{\partial}{\partial x_4},
\end{align*}and let $L,R$ be the corresponding Dirac structures. Then note that
\begin{align*}
 L \diamondsuit_{T^*}R = \{ (\xi,\eta) \ | \ \pi_L^{\sharp}(\xi)=\pi_R^{\sharp}(\eta) \}
\end{align*}is the trivial bundle. This implies by inspection of \eqref{eq : torsion of a pair} that the torsion of $(L,R)$ vanishes identically; that is, $(L,R)$ form a Dirac pair. However, $L$ and $R$ do not concur weakly, since $\pi_L$ and $\pi_R$ do not commute:
\begin{align*}
 [\pi_L,\pi_R] = \tfrac{\partial}{\partial x_2} \wedge \tfrac{\partial}{\partial x_3} \wedge \tfrac{\partial}{\partial x_4}.
\end{align*}
\end{example}

\begin{remark}\normalfont
 It is interesting to highlight that the notion of concurrence involves an altogether different algebraic space from that of Dirac pairs. The example above illustrates our advocacy for (weak) concurrence as the more adequate (even mandated) notion of compatibility between Dirac structures.
\end{remark}

\section{Transverse Dirac structures}\label{sec : Transverse Dirac structures}

The case in which concurrence between Dirac structures $L$ and $R$ has its cleanest interpretation is when these structures are transverse:

\begin{lemma}\label{lem : equivalent conditions for transversality revisited}
For Lagrangian subbundles $L,R \subset \mathbb{T} M$, the following conditions are equivalent:
   \begin{enumerate}[i)]
  \item $L$ and $R$ are transverse: $\mathbb{T} M = L \oplus R$;
  \item $L \star \mathcal{R}_{-1}(R)$ is the graph of a smooth bivector and $L \circledast \mathcal{R}_{-1}(R)$ is the graph of a smooth two-form.
 \end{enumerate}
\end{lemma}
\begin{proof}
First observe that
\begin{align*}
 & \left(L \star \mathcal{R}_{-1}(R)\right) \cap TM = \mathrm{pr}_T(L \cap R), & \left(L \circledast \mathcal{R}_{-1}(R)\right) \cap T^*M = \mathrm{pr}_{T^*}(L \cap R).
\end{align*}
Hence if $L \star \mathcal{R}_{-1}(R)$ is the graph of a smooth bivector, and $L \circledast \mathcal{R}_{-1}(R)$ is the graph of a smooth two-form, it follows that $L \cap R = 0$, and therefore $L$ and $R$ are transverse. Conversely, if $L$ and $R$ are transverse, then
\begin{align*}
 && \mathrm{pr}_T(L)+\mathrm{pr}_T(R)=TM, && \mathrm{pr}_{T^*}(L)+\mathrm{pr}_{T^*}(R)=T^*M.
\end{align*}
By Criteria \ref{star criterion} and \ref{circledast criterion}, we deduce that $L \star \mathcal{R}_{-1}(R)$ and $L \circledast \mathcal{R}_{-1}(R)$ are Lagrangian subbundles, while
\begin{align*}
 && \left(L \star \mathcal{R}_{-1}(R)\right) \cap TM = 0 && \left(L \circledast \mathcal{R}_{-1}(R)\right) \cap T^*M = 0
\end{align*}
imply that $L \star \mathcal{R}_{-1}(R)$ is the graph of a smooth bivector, and $L \circledast \mathcal{R}_{-1}(R)$ is the graph of a smooth two-form.
\end{proof}

\begin{remark}
 Note in particular that, in the context of Lemma \ref{lem : equivalent conditions for transversality revisited}, leaves of $L$ meet leaves of $R$ \emph{transversally}; in particular the dimension of the leaves of $L \star R$ is completely determined by the dimensions of the leaves of $L$ and of $R$.
\end{remark}

One of the important examples of transverse Dirac structures --- and which helped bring attention to Dirac methods in Poisson geometry --- is that of \emph{coupling Dirac structures} \cite{Brahic_Fernandes,Vaisman,Vorobjev}:

\begin{example}[Coupling Dirac structures]\label{ex : coupling}\normalfont
 A Lagrangian subbundle $L \subset \mathbb{T}M$ is {\bf coupling} for a foliation $\mathscr{F}$ on $M$ if $L$ and $\mathrm{Gr}(\mathscr{F})$ are transverse Dirac structures. By Lemma \ref{lem : equivalent conditions for transversality revisited}, this is equivalent to the equalities
 \begin{align*}
  & L \star \mathrm{Gr}(\mathscr{F}) = \mathrm{Gr}(\pi),  & L \circledast \mathrm{Gr}(\mathscr{F}) = \mathrm{Gr}(\omega), 
 \end{align*}where $\pi$ is a vertical Poisson bivector and $\omega$ is a horizontal two-form; in fact,
 \begin{align*}
 & L \star \mathrm{Gr}(\mathscr{F}) = L \cap (N^*\mathscr{F})^{\perp} + N^*\mathscr{F}, & L \circledast \mathrm{Gr}(\mathscr{F}) = L \cap (T\mathscr{F})^{\perp} + T\mathscr{F}.
\end{align*}
Using transversality of $L$, we can write
 \begin{align*}
  & \mathrm{Gr}(\pi) = D \oplus N^*\mathscr{F},  & \mathrm{Gr}(\omega) = C \oplus T\mathscr{F},
 \end{align*}
 where
 \begin{align*}
  & D = L \cap (N^*\mathscr{F})^{\perp},  & C = L \cap (T\mathscr{F})^{\perp},
 \end{align*}
 from which it also follows that
 \begin{align*}
  L = D \oplus C.
 \end{align*}
Note also that the Lagrangian subbundle
 \begin{align*}
  R :=\mathcal{R}_{-\omega}\mathcal{R}_{-\pi}(L) = \mathcal{R}_{-\pi}(D) \oplus \mathcal{R}_{-\omega}(C)
 \end{align*}is such that
 \begin{align*}
  & R \cap TM = \mathcal{R}_{-\omega}(C),  & R \cap T^*M = \mathcal{R}_{-\pi}(D);
 \end{align*}
 therefore
 \begin{align*}
  R = (R \cap TM) \oplus (R \cap T^*M) = \mathrm{Gr}(H)
 \end{align*}is the graph of an Ehresmann connection for $\mathscr{F}$ --- that is, a vector subbundle $H \subset TM$, for which $TM = T\mathscr{F} \oplus H$. In this coupling case, observe that $L$ is completely determined by its Dirac products with $\mathrm{Gr}(\mathscr{F})$:
 \begin{align*}
  L = \mathcal{R}_{\pi}(H^{\circ}) \oplus \mathcal{R}_{\omega}(H).
 \end{align*}
\end{example}

In the sequel, we will discuss another prominent case of Dirac products of transverse Dirac structures --- that of generalized complex structures in Section \ref{sec : Generalized complex structures}. For the moment, let us illustrate the usefulness of this framework with following example of \cite[Corollary 6.7]{Liu_Weinstein_Xu}, which we clarify by recognizing it as an instance of Dirac products (see also \cite{Marcut}):

\begin{proposition}[Transverse Poisson structures]\label{pro : transverse Poisson}
 Let $\pi_L,\pi_R$ be two Poisson structures on a manifold $M$. Then $\mathrm{Gr}(\pi_L)$ and $\mathrm{Gr}(\pi_R)$ are transverse Dirac structures exactly when $\pi_{R}-\pi_{L}$ is the inverse of a nondegenerate two-form $\omega \in \Omega^2(M)$,
 \begin{align*}
  \pi_{R}^{\sharp}-\pi_{L}^{\sharp} = (\omega^{\sharp})^{-1},
 \end{align*}
 in which case
 \begin{align*}
  \mathrm{Gr}(\pi_L) \star \mathrm{Gr}(-\pi_R) = \mathrm{Gr}(\pi_{LR}),
 \end{align*}where $\pi_{LR}$ is the Poisson structure given by
 \begin{align*}
  \pi_{LR}^{\sharp} = \pi_{L}^{\sharp}\omega^{\sharp}\pi_{R}^{\sharp} = \pi_{R}^{\sharp}\omega^{\sharp}\pi_{L}^{\sharp}.
 \end{align*}
\end{proposition}
\begin{proof}
Let $\pi_L,\pi_R$ be two Poisson structures on a manifold $M$. Then $\mathrm{Gr}(\pi_L)$ and $\mathrm{Gr}(\pi_R)$ are transverse exactly when the bivector $\pi_L-\pi_R$ is nondegenerate, i.e., the inverse of a two-form $\omega \in \Omega^2(M)$. Note that
 \begin{align*}
  \pi_L^{\sharp}\omega^{\sharp}\pi_R^{\sharp} - \pi_R^{\sharp}\omega^{\sharp}\pi_L^{\sharp} & =  \pi_L^{\sharp}\omega^{\sharp}(\pi_R^{\sharp} - \pi_L^{\sharp}) + (\pi_L^{\sharp}-\pi_R^{\sharp})\omega^{\sharp}\pi_L^{\sharp} = 0
 \end{align*}shows that $\pi_L^{\sharp}\omega^{\sharp}\pi_R^{\sharp}=\pi_R^{\sharp}\omega^{\sharp}\pi_L^{\sharp}$. Therefore a unique bivector $\pi_{LR} \in \mathfrak{X}^2(M)$ exists, such that
 \begin{align}\label{eq : symmetric composition}
\pi_L^{\sharp}\omega^{\sharp}\pi_R^{\sharp} = \pi_{LR}^{\sharp} = \pi_R^{\sharp}\omega^{\sharp}\pi_L^{\sharp}.
 \end{align}
Next observe that, for all $\xi \in T^*M$,
\begin{align*}
\pi_{LR}^{\sharp}(\xi)+\xi & = \pi_{LR}^{\sharp}(\xi)+\omega^{\sharp}\pi_R^{\sharp}(\xi)-\omega^{\sharp}\pi_L^{\sharp}(\xi)
\end{align*}
lies in $\mathrm{Gr}(\pi_L) \star \mathrm{Gr}(-\pi_R)$ in light of \eqref{eq : symmetric composition} and the fact that
\begin{align*}
 && \pi_L^{\sharp}\omega^{\sharp}\pi_R^{\sharp}(\xi)+\omega^{\sharp}\pi_R^{\sharp}(\xi) \in \mathrm{Gr}(\pi_L), 
 && \pi_R^{\sharp}\omega^{\sharp}\pi_L^{\sharp}(\xi)+\omega^{\sharp}\pi_L^{\sharp}(\xi) \in \mathrm{Gr}(\pi_R).
\end{align*}
This shows that $\mathrm{Gr}(\pi_{LR})$ is contained in the tangent product $\mathrm{Gr}(\pi_L) \star \mathrm{Gr}(-\pi_R)$, and therefore, by a dimension count,
\begin{align*}
\mathrm{Gr}(\pi_{LR}) = \mathrm{Gr}(\pi_L) \star \mathrm{Gr}(-\pi_R).
\end{align*}As a tangent product of Dirac structures, Proposition \ref{pro : automatic involutivity} implies that $\mathrm{Gr}(\pi_{LR})$ is Dirac; that is, that $\pi_{LR}$ is a Poisson structure.
\end{proof}

The dual version of Proposition \ref{pro : transverse Poisson}, in which Poisson structures are replaced by closed two-forms, and the tangent product is replaced by the cotangent product, requires the concurrence hypothesis:

\begin{proposition}[Transverse closed two-forms]
 Let $\omega_L,\omega_R$ be two closed two-forms on a manifold $M$. Then $\mathrm{Gr}(\omega_L)$ and $\mathrm{Gr}(\omega_R)$ are transverse Dirac structures exactly when $\omega_{R}-\omega_{L}$ is the inverse of a nondegenerate bivector $\pi \in \mathfrak{X}^2(M)$,
 \begin{align*}
  \omega_{R}^{\sharp}-\omega_{L}^{\sharp} = (\pi^{\sharp})^{-1},
 \end{align*}
 in which case
 \begin{align*}
  \mathrm{Gr}(\omega_L) \circledast \mathrm{Gr}(-\omega_R) = \mathrm{Gr}(\omega_{LR}),
 \end{align*}where $\omega_{LR}$ is the two-form given by
 \begin{align*}
  \omega_{LR}^{\sharp} = \omega_{L}^{\sharp}\pi^{\sharp}\omega_{R}^{\sharp} = \omega_{R}^{\sharp}\pi^{\sharp}\omega_{L}^{\sharp},
 \end{align*}
 and $\omega_{LR}$ is closed exactly when $\mathrm{Gr}(\omega_L)$ and $\mathrm{Gr}(\omega_R)$ concur.
\end{proposition}

\section{Libermann theorem and a local normal form}

One of the most instructive instances of cotangent products concerns the case in which one of the factors is a foliation $\mathscr{F}$. Our goal in this section is to recast the Dirac version  \cite{Frejlich_Marcut} of Libermann's theorem \cite{Libermann} in terms of concurrence of Dirac structures, and to explain how the very language of Dirac products leads naturally to a local normal form expressing a Dirac structure in terms of its building blocks: closed two-forms, foliations and Poisson structures.

Recall that a foliation $\mathscr{F}$ on a manifold $M$ is {\bf simple} if its leaves are the fibres of a surjective submersion $p:M \to N$. Equivalently, $\mathscr{F}$ is simple if its leaf space
\begin{align*}
 M/\mathscr{F} = \{ S \ | \ S \ \text{leaf of} \ \mathscr{F}\},
\end{align*}
equipped with the quotient topology via the canonical projection
\begin{align*}
& p:M \longrightarrow  M/\mathscr{F}, & p(x)=\mathrm{S}_{\mathscr{F}}(x)
\end{align*}
has the structure of a smooth manifold, for which the quotient map $p$ is a submersion. 

\begin{remark}
 The description of a foliation $\mathscr{F}$ on $M$ via foliated charts implies any foliation is simple in an open neighborhood of a point in $M$.
\end{remark}

\begin{proposition}[Libermann theorem]\label{pro : libermann}
 Let $\mathscr{F}$ be a simple foliation on $M$, and let $L$ be a Dirac structure. Then the following conditions are equivalent:
 \begin{enumerate}[i)]
  \item $L$ and $\mathrm{Gr}(\mathscr{F})$ concur;
  \item There exists a Dirac structure $R$ on $N:={M/}\mathscr{F}$, for which
  \begin{align*}
   p:(M,L) \longrightarrow (N,R)
  \end{align*}
  is forward Dirac: $p_!(L)=R$,
 \end{enumerate}
in which case
\begin{align*}
 L \circledast \mathrm{Gr}(\mathscr{F}) = p^!(R).
\end{align*}
\end{proposition}
\begin{proof}
 (Sketch) Because fibres of $p$ are connected, Dirac structures on $M$ which contain $T\mathscr{F}$ are exactly pullbacks of Dirac structures on $N$, since in that case, involutivity of $L$ implies that
 \begin{align*}
   p_!(L_x) \subset \mathbb{T}_{p(x)}N
 \end{align*}only depends on $p(x) \in N$ (as opposed to $x \in M$). On the other hand, $L \circledast \mathrm{Gr}(\mathscr{F})$ always contains $T\mathscr{F}$, and
 \begin{align*}
   && p_!(L_x) = p_!\left( L \circledast \mathrm{Gr}(\mathscr{F}) \right)_x, && x \in M,
 \end{align*}shows that
 \begin{align*}
  && p_!(L)=R \ \text{is Dirac} && \Longleftrightarrow && L \circledast \mathrm{Gr}(\mathscr{F})=p^!(R) && \Longleftrightarrow &&  L \circledast \mathrm{Gr}(\mathscr{F}) \ \text{is Dirac} .
 \end{align*}
 See \cite[Proposition 2]{Frejlich_Marcut} for further details.
\end{proof}

This leads to the following useful observation:
\begin{lemma}\label{lem : PT form for Dirac structures}
 Let $L$ be a Dirac structure on $M$. A closed two-form $\omega \in \Omega^2(M)$ for which $L \cap \mathrm{Gr}(\omega)$ is a vector subbundle defines a foliation
 \begin{align*}
T\mathscr{F}:=\mathcal{R}_{-\omega}(L \cap \mathrm{Gr}(\omega))  
 \end{align*}
If this foliation is simple, and defined by a surjective submersion $p:M \to N$, then there is a Poisson structure $\pi$ on $N$, such that
 \[
  L = \mathcal{R}_{\omega}p^!\mathrm{Gr}(\pi_N). 
 \]
 \end{lemma}
 \begin{proof}
  The hypotheses imply that the linear family $\mathcal{R}_{-\omega}(L \cap \mathrm{Gr}(\omega))$ is involutive and smooth; it therefore corresponds to a foliation $\mathscr{F}$ on $M$. Moreover, because $T\mathscr{F} \subset \mathcal{R}_{-\omega}(L)$, it follows that
  \begin{align*}
   \mathcal{R}_{-\omega}(L) \circledast \mathrm{Gr}(\mathscr{F})=\mathcal{R}_{-\omega}(L).
  \end{align*}Therefore, by Proposition \ref{pro : libermann}, if $\mathscr{F}$ is defined by surjective submersion $p:M \to N$, there is a Dirac structure $R$ on $N$, such that
  \begin{align*}
   \mathcal{R}_{-\omega}(L) = p^!(R).
  \end{align*}
  Note also that any $a \in \mathcal{R}_{-\omega}(L)$ which is $p$-related to $b \in R$ is automatically vertical, and so $R \cap TN = 0$ --- that is, $R=\mathrm{Gr}(\pi)$ for a Poisson structure $\pi$ on $N$. Therefore
  \begin{align*}
   L = \mathcal{R}_{\omega}p^!\mathrm{Gr}(\pi),
  \end{align*}
  and this concludes the proof.
 \end{proof}

We now give a simple application of the ideas in this note --- namely, we employ cotangent products and concurrence to reprove in rather trivial fashion a version of the local normal form theorem \cite[Theorem 3.2]{Blohmann} (cf. also \cite{Dufour_Wade}) :

\begin{theorem}[Local normal form]\label{thm : Local normal form around points}
 Let $L$ be a Dirac structure on $M$. Then every $x \in M$ has an open neighborhood $U$ on which $L$ takes the form
 \begin{align*}
  L|_U = \mathcal{R}_{\mathrm{d}\alpha}p^!\mathrm{Gr}(\pi),
 \end{align*}
 where $p:U \to N$ is a surjective submersion with connected fibres, $\pi$ is a Poisson structure on $N$, and $\alpha$ is a one-form on $U$.
 \end{theorem}
 \begin{proof}
 Let $x \in M$. Then in an open set $x \in U \subset M$, one can find a one-form $\alpha \in \Omega^1(U)$ for which
 \begin{align*}
  \mathrm{pr}_{T^*}: \mathcal{R}_{-\mathrm{d}\alpha}(L)_x \longrightarrow T^*_xM
 \end{align*}
has maximal rank --- that is, has rank $\dim(M)$ or $\dim(M)-1$, depending on the parity of $L$ alone. This being an open condition, we may assume it holds throughout $U$. In the case where the rank is $\dim(M)$, this already implies that $\mathcal{R}_{-\mathrm{d}\alpha}(L)|_U$ is the graph of a bivector, which is necessarily Poisson. In the case where the rank is $\dim(M)-1$, $\mathcal{R}_{-\mathrm{d}\alpha}(L) \cap TU$ defines a one-dimensional foliation $\mathscr{F}$; shrinking $U$ if need be, we may assume that $\mathscr{F}$ is defined by a surjective submersion $p : U \to N$ with connected fibres, in which case the result follows from Lemma \ref{lem : PT form for Dirac structures}.
\end{proof}

\section{Magri-Morosi conditions}

Our next illustration of how concurrence encompasses classical notions concerns Poisson structures and closed two-forms. As we discuss below, concurrence between a Poisson structure $\pi \in \mathfrak{X}^2(M)$ and a closed two-form $\omega \in \Omega^2(M)$ has various classical, equivalent descriptions, and recover the notion of $P\Omega$-structures \cite{Magri_Morosi} and complementary two-forms of \cite{Vaisman_forms}.

\subsection*{Twisted brackets}

An endomorphism $a:TM \to TM$ gives rise to a {\bf twisted} bracket
\begin{align*}
 & {[\cdot,\cdot]}^a : \mathfrak{X}(M) \times \mathfrak{X}(M) \to \mathfrak{X}(M), & [u,v]^a:=[a(u),v]+[u,a(v)]-a[u,v].
\end{align*}
It will also be convenient to consider endomorphisms $b:\mathbb{T}M \to \mathbb{T}M$ and twists of the Dorfman bracket:
\begin{align*}
 {[\cdot,\cdot]}^b  &: \Gamma(\mathbb{T}M) \times \Gamma(\mathbb{T}M) \to \Gamma(\mathbb{T}M),\\
 [u+\xi,v+\eta]^b & :=[b(u+\xi),v+\eta]+[u+\xi,b(v+\eta)]-b[u+\xi,v+\eta].
\end{align*}
Regarding a bivector $\pi \in \mathfrak{X}^2(M)$ as an endomorphism of $\mathbb{T}M$, we have that
\begin{align*}
 & {[\cdot,\cdot]}^{\pi} : \Omega^1(M) \times \Omega^1(M) \to \Omega^1(M), & [\xi,\eta]^{\pi}:=\mathscr{L}_{\pi(\xi)}\eta-\iota_{\pi(\eta)}\mathrm{d}\xi
\end{align*}
is the Koszul bracket of $\pi$. Regarding a two-form $\omega \in \Omega^2(M)$ as an endomorphism of $\mathbb{T}M$, we have that
\begin{align*}
  & {[\cdot,\cdot]}^{\omega} : \mathfrak{X}(M) \times \mathfrak{X}(M) \to \Omega^1(M), & [u,v]^{\omega}:=[\omega(u),v]+[u,\omega(v)]-\omega[u,v]
\end{align*}
coincides with $\mathrm{d}\omega(u,v)$.

\begin{remark}
 In the formulas above, as in the remainder of this section, we neglect to write the superscript $\sharp$ in the linear maps $\pi^{\sharp}:T^*M \to TM$ and $\omega^{\sharp}:TM \to T^*M$ in order not to overburden the notation. 
\end{remark}

\subsection*{Nijenhuis endomorphisms}
The {\bf Nijenhuis torsion} of an endomorphism $a:TM \to TM$ is given by
\begin{align*}
 & \mathrm{N}(a):\mathfrak{X}(M) \times \mathfrak{X}(M) \to \mathfrak{X}(M), & \mathrm{N}(a)(u,v):=a[u,v]^a-[a(u),a(v)]
\end{align*}
and we say that $a$ is {\bf Nijenhuis} if $\mathrm{N}(a)=0$. In that case, $(TM,{[\cdot,\cdot]}^a,a)$ is a Lie algebroid structure on $TM$; that is,
\begin{align*}
 & \left[u,[v,w]^a\right]^a = \left[[u,v]^a,w\right]^a+\left[v,[u,w]^a\right]^a, & [u,fv]^a=f[u,v]^a+(\mathscr{L}_{a(u)}f)v
\end{align*}
hold for all $u,v,w \in \mathfrak{X}(M)$ and all $f \in C^{\infty}(M)$. Moreover, if $a$ is Nijenhuis, then so is any power of $a$.

\subsection*{$PN$-structures}

Let $a:TM \to TM$ be a endomorphism, and $\pi\in \mathfrak{X}^2(M)$ be a bivector. We say that $\pi$ is $a$-symmetric if
\begin{align*}
 \pi(a^*(\xi),\eta) = \pi(\xi,a^*(\eta)), & \xi,\eta \in \Omega^1(M).
\end{align*}
If that is the case, then $a^i\pi$ are bivectors. The {\bf concomitant} of a endomorphism $a$ and an $a$-symmetric bivector $\pi$ is
\begin{align}\label{eq : concomitant a pi}
 \mathrm{C}(a,\pi)(\xi,\eta):=a^*[\xi,\eta]^{\pi} - \left([a^*(\xi),\pi(\eta)]+[\pi(\xi),a^*(\eta)]\right).
\end{align}
When $\pi$ is Poisson, $a$ is Nijenhuis, and the concomitant \eqref{eq : concomitant a pi} of $(\pi,a)$ vanishes, we say that $(\pi,a)$ is a {\bf $PN$-structure}, in which case the bivectors $\pi_i$ are all Poisson and commute pairwise.

\subsection*{$\Omega N$-structures}

Let $a:TM \to TM$ be a endomorphism, and $\omega \in \Omega^2(M)$ be a two-form. We say that $\omega$ is $a$-symmetric if
\begin{align*}
 & \omega(a(u),v) = \omega(u,a(v)), & u,v \in \mathfrak{X}(M).
\end{align*}
If that is the case, then $\omega a^i$ are two-forms. The {\bf concomitant} of an endomorphism $a$ and an $a$-symmetric two-form $\omega$ is the tensor $\mathrm{C}(a,\omega) \in \Gamma(\wedge^2T^*M \otimes T^*M)$ given by
\begin{align}\label{eq : concomitant a omega}
 \mathrm{C}(a,\omega)(u,v):=[a(u),\omega(v)]+[\omega(u),a(v)]-\omega[u,v]^a + a^*[u,v]^{\omega}.
\end{align}
If an $a$-symmetric two-form $\omega$ is closed, $a$ is Nijenhuis, the concomitant \eqref{eq : concomitant a omega} above vanishes exactly when the two-forms $\omega_i \in \Omega^2(M)$ are all closed, in which case we say that $(\omega,a)$ is a {\bf $\Omega N$-structure}. 

\subsection*{$P\Omega$-structures}

Let $\pi\in \mathfrak{X}^2(M)$ be a Poisson structure and $\omega \in \Omega^2(M)$ be a closed two-form. Then note that both $\pi$ and $\omega$ are symmetric with respect to the endomorphism $a:=\pi\omega:TM \to TM$, and that the concomitant \eqref{eq : concomitant a pi} of $\pi$ and $a$ vanishes, 
\begin{align*}
 \mathrm{C}(a,\pi)(\xi,\eta) & =\omega\pi[\xi,\eta]^{\pi} - ([\omega\pi(\xi),\pi(\eta)]+[\pi(\xi),\omega\pi(\eta)])\\
 & =\omega[\pi(\xi),\pi(\eta)] - ([\omega\pi(\xi),\pi(\eta)]+[\pi(\xi),\omega\pi(\eta)])\\
 & = 0,
\end{align*}
while the concomitant \eqref{eq : concomitant a omega} of $\omega$ and $a$ vanishes exactly when $\omega\pi\omega$ is a closed two-form:
\begin{align*}
 \mathrm{C}(a,\omega)(u,v) & = [\pi\omega(u),\omega(v)]+[\omega(u),\pi\omega(v)] - \omega[u,v]^{\pi\omega}\\
 & = \omega[\pi\omega(u),v]-[\omega\pi\omega(u),v]+\omega[u,\pi\omega(v)]-[u,\omega\pi\omega(v)]- \omega[u,v]^{\pi\omega}\\
 & = \omega\pi\omega[u,v]+\omega[u,v]^{\pi\omega}-\omega\pi\omega[u,v]-[u,v]^{\omega\pi\omega}- \omega[u,v]^{\pi\omega}\\
 & = - [u,v]^{\omega\pi\omega},
\end{align*}
in which case $a$ is a Nijenhuis endomorphism:
\begin{align*}
 \mathrm{N}(\pi\omega)(u,v) & = \pi\omega[u,v]^{\pi\omega}-[\pi\omega(u),\pi\omega(v)]\\
 & = \pi\left(\omega[u,v]^{\pi\omega}-[\omega(u),\omega(v)]^{\pi}\right)\\
 & = \pi[u,v]^{\omega\pi\omega}.
\end{align*}
We say that $(\pi,\omega)$ is a {\bf $P\Omega$-structure} if $\omega\pi\omega \in \Omega^2(M)$ is closed; equivalently, $(\pi,\omega)$ is a $P\Omega$-structure if 
\begin{align*}
 (\pi,a) \ \ \text{is a $PN$-structure and} \  (\omega,a) \ \text{is a $\Omega N$-structure.}
\end{align*}

\begin{example}\normalfont
 On $M:=\mathbb{C}^n$, with complex coordinates $z_i=x_i+iy_i$, consider the Poisson structure and the closed two-form
 \begin{align*}
  & \pi=\sum_{i=1}^nr_i^2\tfrac{\partial}{\partial x_i} \wedge \tfrac{\partial}{\partial y_i}, & \omega=-\sum_{i=1}^n \mathrm{d} x_i \wedge \mathrm{d} y_i,
 \end{align*}
 where $r_i^2=x_i^2+y_i^2$. Then $(\pi,\omega)$ is a $P\Omega$-structure:
 \begin{align*}
  a = \pi\omega = \sum_{i=1}^n r_i^2 \left(\mathrm{d} x_i \otimes \tfrac{\partial}{\partial x_i}+\mathrm{d} y_i \otimes \tfrac{\partial}{\partial y_i} \right)
 \end{align*}
is Nijenhuis, $(\pi,a)$ is a $PN$-structure, and $(\omega,a)$ is a $\Omega N$-structure, and
 \begin{align*}
  & \pi_n=\sum_{i=1}^nr_i^{2n+2}\tfrac{\partial}{\partial x_i} \wedge \tfrac{\partial}{\partial y_i}, & \omega_n=-\sum_{i=1}^n r_i^{2n}\mathrm{d} x_i \wedge \mathrm{d} y_i.
 \end{align*}
\end{example}

\subsection*{Two-forms complementary to Poisson structures}

There is yet another compatibility condition between forms and Poisson structures. Let $A$ denote the Lie algebroid $(T^*M,{[\cdot,\cdot]}^{\pi},\pi)$ associated to a Poisson structure $\pi \in \mathfrak{X}(M)$. Then a two-form $\omega \in \Omega^2(M)$ can be interpreted as a bivector $\omega \in \Gamma(\wedge^2A)$ on the Lie algebroid $A$. The condition that $\omega$ be Poisson in $A$ reads
\begin{align*}
 & \omega[u,v]^{\pi^{\omega}} = [\omega(u),\omega(v)]^{\pi} , & u,v \in \mathfrak{X}(M) = \Gamma(A^*).
\end{align*}
If that is the case, we say that $\omega$ is a {\bf complementary} two-form to the Poisson structure $\pi$.

\begin{proposition}\label{pro : equivalent POmega}
Let $\pi \in \mathfrak{X}^2(M)$ a Poisson structure and $\omega \in \Omega^2(M)$ a closed two-form. Let $a$ be the endomorphism of $TM$ defined by $a:=\pi^{\sharp}\omega^{\sharp}$. Then the following are equivalent:
 \begin{enumerate}[i)]
  \item $\mathrm{Gr}(\pi)$ and $\mathrm{Gr}(\omega)$ concur;
  \item $(\pi,\omega)$ is a $P\Omega$-structure;
  \item $\omega$ is a complementary two-form to $\pi$.
 \end{enumerate}
\end{proposition}
\begin{proof}
Introduce the tensor
\begin{align*}
        & \mathrm{R}_{\omega} : \mathfrak{X}(M) \times \mathfrak{X}(M) \to \Omega^1(M), & \mathrm{R}_{\omega}(u,v):=\omega[u,v]^{a}-[\omega(u),\omega(v)]^{\pi}
\end{align*}
and consider the Lagrangian family
    \begin{align*}
        & \mathrm{Gr}(\pi) \circledast \mathrm{Gr}(\omega) = \{s_u \ | \ u \in TM\}, & s_u:=b(u)+\omega(u),
    \end{align*}where $b:TM \to TM$ denotes the endomorphism $b:=\mathrm{id}+a$. This is patently a smooth structure, and it is Dirac iff
    \begin{align*}
        \Upsilon(s_u,s_v,s_w):=\langle [s_u,s_v],s_w\rangle
    \end{align*}vanishes for all $u,v,w \in \mathfrak{X}(M)$. Now,
    \begin{align*}
        \Upsilon(s_u,s_v,s_w)& = \langle [b(u),b(v)],\omega(w)\rangle + \langle [b(u),\omega(v)]+[\omega(u),b(v)],b(w) \rangle \\
        & = \langle b^*([b(u),\omega(v)]+[\omega(u),b(v)])-\omega[b(u),b(v)],w \rangle.
    \end{align*}
    On the one hand,
    \begin{align*}
        b^*([b(u),\omega(v)]+[\omega(u),b(v)]) & = b^*([u,\omega(v)]+[\omega(u),v]+[a(u),\omega(v)]+[\omega(u),a(v)])\\
        & = b^*(\omega[u,v]+[\omega(u),\omega(v)]^{\pi})\\
        & = \omega[u,v]+[\omega(u),\omega(v)]^{\pi}+\omega a[u,v]+a^*[\omega(u),\omega(v)]^{\pi}
    \end{align*}
    while on the other
    \begin{align*}
        \omega[b(u),b(v)] & = \omega[u,v]+\omega[a(u),v]+\omega[u,a(v)]+\omega[a(u),a(v)] \\
        & = \omega[u,v]+\omega[u,v]^{a}+\omega a[u,v]+a^*[\omega(u),\omega(v)]^{\pi}
    \end{align*}and therefore
    \begin{align*}
        \Upsilon(s_u,s_v,s_w) & = \langle [\omega(u),\omega(v)]^{\pi}-\omega[u,v]^{a},w \rangle \\
        & = -\langle \mathrm{R}_{\omega}(u,v),w \rangle
    \end{align*}
    Hence $\mathrm{Gr}(\pi)$ and $\mathrm{Gr}(\omega)$ concur exactly when $\mathrm{R}_{\omega}$ vanishes identically. Next observe that
\begin{align*}
 [u,v]^{\omega\pi\omega} & = [\omega\pi\omega(u),v]+[u,\omega\pi\omega(v)]-\omega\pi\omega[u,v]\\
 & = \omega[\pi\omega(u),v]-[\pi\omega(u),\omega(v)] + \omega[u,\pi\omega(v)]-[\omega(u),\pi\omega(v)]-\omega\pi\omega[u,v]\\
 & = \omega\pi\omega[u,v]+\omega[u,v]^{\pi\omega}-[\omega(u),\omega(v)]^{\pi}-\omega\pi\omega[u,v]\\
 & = \mathrm{R}_{\omega}(u,v)
\end{align*}
shows that $\omega\pi\omega$ is closed exactly when $\mathrm{R}_{\omega}=0$. Hence i) and ii) are equivalent. Finally, note that
\begin{align*}
 [u,v]^{\pi^{\omega}} & = [\omega(u),v]^{\pi}+[u,\omega(v)]^{\pi} \\
 & = [\pi\omega(u),v]-\pi[\omega(u),v]+[u,\pi\omega(v)]-\pi[u,\omega(v)]\\
 & = \pi\omega[u,v]+[u,v]^{\pi\omega}-\pi\omega[u,v]\\
 & = [u,v]^{\pi\omega}
\end{align*}
implies that
\begin{align*}
 \omega[u,v]^{\pi^{\omega}} = [\omega(u),\omega(v)]^{\pi}
\end{align*}
holds true exactly when $\mathrm{R}_{\omega}=0$. Hence i), ii) and iii) are equivalent.
\end{proof}

\section{Complex Dirac structures}\label{sec : Complex Dirac structures}

We consider the version of Courant geometry in which scalars are extended from real- to complex numbers. More explicitly, by
\begin{align*}
 \mathbb{T}_{\mathbb{C}}M:=\mathbb{T}M \otimes_{\mathbb{R}}\mathbb{C}
\end{align*}we denote the complex Courant algebroid of $M$, and by
\begin{align}\label{eq : complex extension}
 & \langle \cdot,\cdot\rangle_{\mathbb{C}} : \mathbb{T}_{\mathbb{C}}M \times_M\mathbb{T}_{\mathbb{C}}M \to \mathbb{C}, & [ \cdot,\cdot ]_{\mathbb{C}} : \Gamma(\mathbb{T}_{\mathbb{C}}M) \times \Gamma(\mathbb{T}_{\mathbb{C}}M) \to \Gamma(\mathbb{T}_{\mathbb{C}}M)
\end{align}the extension of \eqref{eq : symmetric bilinear pairing} and \eqref{eq : dorfman bracket} to $\mathbb{C}$-bilinear maps. We call \emph{complex linear families} subspaces $L \subset \mathbb{T}_{\mathbb{C}}M$ which meet every fibre of $\mathbb{T}_{\mathbb{C}}M$ in a complex-linear subspace. A complex linear family $L$ induces an orthogonal complex linear family
\begin{align*}
 L^{\perp} = \{ a \in  \mathbb{T}_{\mathbb{C}}M \ | \ \langle a,\cdot\rangle_{\mathbb{C}}|_L=0 \}
\end{align*}
and a complex linear family is \emph{Lagrangian} if $L=L^{\perp}$. 

\begin{example}
 If $L$ is a Lagrangian family on $M$, then its {\bf scalar extension}
 \begin{align*}
  L \otimes \mathbb{C} \subset \mathbb{T}_{\mathbb{C}}M
 \end{align*}
 is a complex Lagrangian family on $M$.
\end{example}

\begin{lemma}
 For all complex Lagrangian families $L,R \subset \mathbb{T}_{\mathbb{C}}M$, we have that
 \begin{align*}
  & \overline{L \star R} = \overline{L} \star \overline{R}, & \overline{L \circledast R} = \overline{L} \circledast \overline{R}.
 \end{align*}
 In particular, for every complex Lagrangian family $L \subset \mathbb{T}_{\mathbb{C}}M$, the complex Lagrangian families
 \begin{align*}
  & L \star \overline{L}, & L \circledast \overline{L}.
 \end{align*}are the scalar extension of real Lagrangian families $L_{\star},L_{\circledast} \subset \mathbb{T}M$
 \begin{align*}
  & L_{\star} \otimes \mathbb{C} = L \star \overline{L}, & L_{\circledast} \otimes \mathbb{C} = L \circledast \overline{L}.
 \end{align*}
\end{lemma}
\begin{proof}
 Compute
 \begin{align*}
  \overline{L \star R} & = \{ \overline{a} \ | \ a \in L \star R\}\\
  & = \{ \overline{a} \ | \ a = a_L+\mathrm{pr}_{T^*}(a_R), \ \mathrm{pr}_{T}(a_L-a_R)=0\}\\
  & = \{ \overline{a} \ | \ \overline{a} = \overline{a_L}+\mathrm{pr}_{T^*}(\overline{a_R}), \ \mathrm{pr}_{T}(\overline{a_L}-\overline{a_R})=0\}\\
  & = \overline{L} \star \overline{R}.
 \end{align*}
 Therefore
 \begin{align*}
  \overline{L \star \overline{L}} = \overline{L} \star L = L \star \overline{L}
 \end{align*}is the complexification of
 \begin{align*}
  & L_{\star} \otimes \mathbb{C} = L \star \overline{L}, & L_{\star}:=(L \star \overline{L}) \cap \mathbb{T}M.
 \end{align*}
The proof for $\circledast$ is symmetric.
\end{proof}

A {\bf complex Dirac structure} is a complex vector subbundle $L \subset \mathbb{T}_{\mathbb{C}}M$ which is Lagrangian, $L=L^{\perp}$, and whose space of sections $\Gamma(L)$ is involutive under the Dorfman bracket --- equivalently, whose Courant tensor
\begin{align*}
 & \Upsilon_L \in \Gamma(\wedge^3L^*), & \Upsilon_L(a_1,a_2,a_3):=\langle [a_1,a_2]_{\mathbb{C}},a_3\rangle_{\mathbb{C}}
\end{align*}
vanishes identically. When that is the case, $L$ becomes a complex Lie algebroid with anchor $\mathrm{pr}_T:L \to T_{\mathbb{C}}M$ and bracket given by the restriction of the complex Dorfman bracket.

\begin{remark}\normalfont
 A crucial difference between the real and the complex cases is that complex Dirac structures (or complex Lie algebroids, for that matter) do \underline{not} induce a partition of the ambient manifold into leaves/orbits. At its core, the issue is that we do not understand what the ``flow'' of a complex vector field ought to be, or if it makes sense at all. See \cite{Weinstein_complex} for a detailed discussion.
\end{remark}

\begin{example}[Scalar extensions]
The scalar extension $L:= R \otimes \mathbb{C} \subset \mathbb{T}_{\mathbb{C}}M$ of a (usual) Dirac structure $R \subset \mathbb{T} M$. is a complex Dirac structure. Those complex Dirac structures which arise in this manner are those which are invariant under conjugation: $L=\overline{L}$.
\end{example}

\begin{example}[Involutive structures]\label{ex : IS}
 Let $E \subset T_{\mathbb{C}}M$ be an \emph{involutive structure} --- that is, a complex vector subbundle which is involutive under $[\cdot,\cdot]_{\mathbb{C}}$. Then $\mathrm{Gr}(E)=E\oplus E^{\circ}$ is a complex Dirac structure. 
\end{example}

\begin{example}[Generalized complex structures]\label{ex : GCS}
Let $J:\mathbb{T}M \to \mathbb{T}M$ be an orthogonal almost complex structure:
\begin{align*}
 & J^*\langle \cdot,\cdot\rangle = \langle \cdot,\cdot \rangle, & J^2+\mathrm{id}=0.
\end{align*}Its scalar extension $J_{\mathbb{C}}:\mathbb{T}_{\mathbb{C}}M \to \mathbb{T}_{\mathbb{C}}M$ decomposes $\mathbb{T}_{\mathbb{C}}M$ into $\pm i$-eigenbundles $\mathbb{T}_{\mathbb{C}}M=L \oplus \overline{L}$. We say that $J$ is a {\bf generalized complex structure} if $L$ is a complex Dirac structure. A complex Dirac structures arises in this manner exactly if it is transverse to its conjugate: $L \cap \overline{L}=0$.
\end{example}

\begin{remark}\normalfont
 In the case where Criterion \ref{star criterion} applies, the authors of \cite{Aguero_Rubio} call the ``associated real Dirac structure'' the Dirac structure whose scalar extension is
 \begin{align*}
  \widehat{L} = \mathcal{R}_2(L \star \overline{L}).
 \end{align*}
 With our methods, the proof of \cite[Theorem 5.1]{Aguero_Rubio} reduces to Criterion \ref{star criterion} itself.
\end{remark}

\begin{example}\normalfont
Elaborating on a comment in \cite[Section 5]{Aguero_Rubio}, Criteria \ref{star criterion} and \ref{circledast criterion} are not necessary for the Lagrangian families
 \begin{align*}
  & L \star \overline{L}, & L \circledast \overline{L}
 \end{align*}to be the scalar extension of Dirac structures. For example, if $\pi_1 \in \mathfrak{X}^2(M_1)$ and $\pi_2 \in \mathfrak{X}^2(M_2)$ are arbitrary Poisson structures, define on $M:=M_1 \times M_2$ the complex Dirac structure
 \begin{align*}
 L = \mathrm{Gr}(\pi_1) \times \mathrm{Gr}(i\pi_2).
 \end{align*}Then
 \begin{align*}
  & L \star \overline{L} = \mathrm{Gr}(\tfrac{1}{2}\pi_1) \times T_{\mathbb{C}}M_2, & L \circledast \overline{L} = \mathrm{Gr}(2\pi_1) \times T^*_{\mathbb{C}}M_2
 \end{align*}are the scalar extensions of the Dirac structures
 \begin{align*}
  & L_{\star} = \mathrm{Gr}(\tfrac{1}{2}\pi_1) \times TM_2, & L_{\circledast} = \mathrm{Gr}(2\pi_1) \times T^*M_2,
 \end{align*}
 even though Criteria \ref{star criterion} and \ref{circledast criterion} cannot be invoked if either $\pi_1$ or $\pi_2$ does not have constant rank.
\end{example}

\section{Involutive structures and the Frobenius-Nirenberg theorem}

As mentioned in Example \ref{ex : IS}, an involutive structure on $M$ is a complex vector subbundle 
$E\subset T_{\mathbb{C}} M$ whose sections are closed under the complexified Lie bracket $[\cdot,\cdot]_\mathbb{C}$. Before explaining what role concurrence plays within involutive structures, we discuss some particular instances of involutive structures, with a special emphasis on existence of adapted coordinates.

\subsection*{Foliations}

If $\mathscr{F}$ is a foliation on $M$, then
\begin{align*}
 E:=T\mathscr{F} \otimes \mathbb{C} \subset T_{\mathbb{C}}M
\end{align*}
is an involutive structure on $M$. Involutive structures $E$ of this form are exactly those for which $E = \overline{E}$.

\subsection*{Complex structures}

Let $M$ be a complex manifold --- that is, a manifold equipped with an atlas $\mathfrak{A}=\{(U_i,\phi_i)\}$, where $U_i \subset M$ form an open covering, and $\phi_i : U_i \to \mathbb{C}^m$ are open embeddings, for which the transition maps
\begin{align*}
 \phi_{ij}:=\phi_i\phi_j^{-1}:\phi_j(U_i\cap U_j) \to \phi_i(U_i \cap U_j)
\end{align*}
are holomorphic. Then there is an induced involutive structure
\begin{align*}
 E := T^{0,1}M \subset T_{\mathbb{C}}M,
\end{align*}
which satisfies
\begin{align}\label{eq : acomplex as decomposition}
 T_{\mathbb{C}}M = E \oplus \overline{E}.
\end{align}
Any complex vector bundle $E \subset T_{\mathbb{C}}M$ satisfying \eqref{eq : acomplex as decomposition} is the $-i$-eigenbundle of a unique almost-complex structure
\begin{align*}
 & J:TM \to TM, & J^2=-\mathrm{id},
\end{align*}
and $E$ is involutive exactly when the Nijenhuis torsion of $J$ vanishes:
\begin{align*}
 [u,Jv]+[Ju,v]=J\left([u,v]-[Ju,Jv]\right).
 \end{align*}
The famous Newlander-Nirenberg theorem \cite{NN} can be rephrased as follows: if $E$ is an involutive structure satisfying \eqref{eq : acomplex as decomposition}, then around any point $x \in M$ there is a chart $\phi : U \to \mathbb{C}^m$ with the property that
\begin{align*}
 \phi_*(E|_U) = T^{0,1}\phi(U),
\end{align*}
and the collection of all such charts forms a complex-analytic atlas for $M$. Otherwise said:
\begin{lemma}
There is a bijective correspondence between:
\begin{enumerate}[a)]
 \item Complex structures on $M$ compatible with its smooth structure;
 \item Involutive structures $E \subset T_{\mathbb{C}}M$ with $T_{\mathbb{C}}M=E \oplus \overline{E}$;
 \item Almost-complex structures $J:TM \to TM$ with vanishing Nijenhuis torsion.
\end{enumerate}
\end{lemma}

\subsection*{CR structures}

An involutive structure $E \subset T_{\mathbb{C}}M$ with the property that
\begin{align*}
 E \cap \overline{E} = 0
\end{align*}
is called a {\bf CR structure} \cite{Chirka} on $M$ --- short at once for \emph{Cauchy-Riemann} or \emph{Complex-Real}. Complex structures are one source of examples. Another natural way in which these structures arise is the following: suppose $X$ is a complex manifold, and $M \subset X$ is an embedded real submanifold. Then
\begin{align*}
 & E \subset T_{\mathbb{C}}M, & E:=T^{0,1}X \cap T_{\mathbb{C}}M
\end{align*}
is a CR structure on $M$, provided that $E$ have constant rank. That is always the case is $M$ is a real hypersurface in $X$. Note also that, by (the evident scalar extension of) Proposition \ref{pro : libermann}, we have the following fact, which we record as a lemma:

\begin{lemma}\label{lem : inv pushes to CR}
 If an involutive structure $E \subset T_{\mathbb{C}}M$ is such that $E \cap \overline{E}$ corresponds to a simple foliation $\mathscr{F}$ on $M$, given by the fibres of the surjective submersion
\begin{align*}
 & p:M \to N, & \ker p_* = T\mathscr{F},
\end{align*}
then $p_*(E) \subset T_{\mathbb{C}}N$ defines a CR structure on $N$.
\end{lemma}

\subsection*{Foliations with complex leaves} Suppose $\mathscr{F}$ is a foliation on $M$. A complex subbundle $E \subset T\mathscr{F} \otimes \mathbb{C}$ for which
\begin{align*}
 T\mathscr{F} \otimes \mathbb{C} = E \oplus \overline{E}
\end{align*}
is the $-i$-eigenbundle of a unique almost-complex structure $J:T\mathscr{F} \to T\mathscr{F}$. The condition that $E$ be involutive is equivalent to the vanishing of the Nijenhuis torsion of $J$. Otherwise said,

\begin{lemma}
 If a CR structure $E$ on $M$ is such that $E+\overline{E}$ is the scalar extension of the tangent bundle to a foliation, then the leaves of this foliation have the structure of complex manifolds.
\end{lemma}
Such involutive structures  --- also called \emph{Levi-flat CR structures} --- correspond to atlases $\mathfrak{A}=\{(U_i,\phi_i)\}$, where $\phi_i : U_i \to \mathbb{R}^q \times \mathbb{C}^n$, for which the transition functions are of the form \cite[Section 5]{Chirka}
\begin{align*}
 & \phi_{ij}:\phi_j(U_i \cap U_j) \to \phi_i(U_i \cap U_j), & \phi_{ij}(x,z) = (a(x),b(x,z)),
\end{align*}
where $b(x,\cdot)$ is holomorphic for each $x$. 

\subsection*{Transversally holomorphic foliations} Foliations $\mathscr{F}$ of (real) codimension $q$ on $M$ can be described in H\ae fliger fashion as follows. In a sufficiently small open neighborhood $U$ of a point $x \in M$, there is a submersion $s:U \to \mathbb{R}^q$ with $T\mathscr{F}|_U=\ker s_*$.  If $s':V \to \mathbb{R}^q$ is another such submersion, then
\begin{align*}
 s = \gamma s'
\end{align*}
for a unique map $\gamma$ of $U \cap V$ into germs of local diffeomorphisms of $\mathbb{R}^q$. A foliation of real codimension $2q$ presented by local submersions $s:U \to \mathbb{C}^q$ is {\bf transversally holomorphic} \cite{Gomez-Mont} if the maps $\gamma$ take values in germs of local biholomorphisms of $\mathbb{C}^q$. In that case, there is a corresponding involutive structure $E \subset T_{\mathbb{C}}M$, uniquely determined by the property that, for any defining submersion $s:U \to \mathbb{C}^q$, we have that $E|_U = s_*^{-1}(T^{0,1}\mathbb{C}^q)$. Observe that
\begin{align*}
 T\mathscr{F} \otimes \mathbb{C} = E \cap \overline{E}.
\end{align*}
Such foliations correspond to atlases $\mathfrak{A}=\{(U_i,\phi_i)\}$, where $\phi_i : U_i \to \mathbb{C}^q \times \mathbb{R}^n$, for which the transition functions are of the form
\begin{align*}
 & \phi_{ij}:\phi_j(U_i \cap U_j) \to \phi_i(U_i \cap U_j), & \phi_{ij}(z,y) = (a(z),b(z,y)),
\end{align*}
where $a$ is holomorphic. 

\subsection*{Nirenberg structures} Fix integers $n,d$ for which $2n+d \leqslant \dim M$. We call a {\bf Nirenberg structure} of type $(n,d)$ a maximal atlas $\mathfrak{A}=\{(U_i,\phi_i)\}$ of $M$, where the local charts are
\begin{align*}
 \phi_i : U_i \to \mathbb{R}^{\dim(M)-2n-d} \times \mathbb{C}^n \times \mathbb{R}^d,
\end{align*}
and the transition functions take the form
\begin{align*}
 & \phi_{ij}:\phi_j(U_i \cap U_j) \to \phi_i(U_i \cap U_j), & \phi_{ij}(x,z,y) = (a(x),b(x,z),c(x,z,y)),
\end{align*}
where $b(x,\cdot)$ is holomorphic for each $x$. 

Note that there is then a unique involutive structure $E \subset T_{\mathbb{C}}M$, such that
\begin{align*}
 E|_{U_i} = {\phi_i}_*^{-1}\left(\mathbb{R}^{\dim(M)-2n-d} \times T^{0,1}\mathbb{C}^n \times T_{\mathbb{C}}\mathbb{R}^d\right),
\end{align*}
and $E$ is such that
\begin{align*}
 & E \cap \overline{E} = T\mathscr{F}_{\star} \otimes \mathbb{C}, & E + \overline{E} = T\mathscr{F}_{\circledast} \otimes \mathbb{C}
\end{align*}
are the scalar extension of the tangent bundles to foliations $\mathscr{F}_{\circledast}$ and $\mathscr{F}_{\star}$ on $M$, locally described by
\begin{align*}
 & x = \ \text{constant}, & x = \ \text{constant}, \ z = \ \text{constant}, 
\end{align*}
respectively. Observe also that the holomorphic condition on the atlas ensures that the foliation $\mathscr{F}_{\star}$ induces on the leaves of the foliation $\mathscr{F}_{\circledast}$ a transversally holomorphic foliation. 

Observe also that most of our examples fit into this setting:
\begin{itemize}
 \item foliations are Nirenberg structures of type $(0,d)$;
 \item complex structures are Nirenberg structures of type $(\dim(M)/2,0)$;
 \item foliations with complex leaves are Nirenberg structures of type $(n,0)$;
 \item transversally holomorphic foliations are Nirenberg structures of type $(n,\dim(M)-2n)$.
\end{itemize}

In the language of concurrence, the main result of \cite{Nirenberg} can be rephrased as follows:

\begin{proposition}[Frobenius-Nirenberg theorem]
 For an involutive structure $E \subset T_{\mathbb{C}}M$, the following conditions are equivalent:
 \begin{enumerate}[i)]
  \item $\mathrm{Gr}(E)$ and $\mathrm{Gr}(\overline{E})$ concur;
  \item $E$ is a Nirenberg structure of type $(n,d)$ on $M$, where
  \begin{align*}
   & d=\mathrm{rank}_{\mathbb{C}}(E \cap \overline{E}), & n=\mathrm{rank}_{\mathbb{C}}(E)-d.
  \end{align*}
 \end{enumerate}
\end{proposition}
\begin{proof}
 The condition that $\mathrm{Gr}(E)$ and $\mathrm{Gr}(\overline{E})$ concur in particular implies that 
 \begin{align*}
  \mathrm{Gr}(E) \star \mathrm{Gr}(\overline{E}) = \mathrm{Gr}(E \cap \overline{E})
 \end{align*}is smooth, and is hence the scalar extension of a foliation $\mathscr{F}_{\star}$ on $M$. Suppose $\mathscr{F}_{\star}$ is a simple foliation given by the fibres of a surjective submersion
 \begin{align*}
  & p:M \to N, & \ker p_*=T\mathscr{F}_{\star}.
 \end{align*}
 By Lemma \ref{lem : inv pushes to CR}, $p_!\mathrm{Gr}(E) = \mathrm{Gr}(E')$ for a unique CR structure $E'\subset T_{\mathbb{C}}N$, and $\mathrm{Gr}(E)$ and $\mathrm{Gr}(\overline{E})$ concur exactly when $\mathrm{Gr}(E')$ and $\mathrm{Gr}(\overline{E'})$ concur. That is the case exactly when
 \begin{align*}
  T\mathscr{F}'\otimes \mathbb{C} = E' \oplus \overline{E'}
 \end{align*}
 for a foliation $\mathscr{F}'$ on $N$ with complex leaves, and
 \begin{align*}
  & T\mathscr{F}_{\circledast} \otimes \mathbb{C} = E + \overline{E}, & T\mathscr{F}_{\circledast}:=p_*^{-1}(T\mathscr{F}').
 \end{align*}
 Adapted coordinates $(x,z,y)$ for $E$ on $M$ can then be built out of adapted coordinates $(x,z)$ for $E'$ on $N$; observe that, in this coordinate system, $E$ is spanned by
 \begin{align*}
  \tfrac{\partial}{\partial \overline{z}_1}, \ \ \tfrac{\partial}{\partial \overline{z}_2}, \ \ \cdots \ \ \tfrac{\partial}{\partial \overline{z}_{n-d}}, \ \ \tfrac{\partial}{\partial y_1}, \ \ \tfrac{\partial}{\partial y_2}, \ \ \cdots \ \  \tfrac{\partial}{\partial y_d}.
 \end{align*}
\end{proof}

\section{Generalized complex structures}\label{sec : Generalized complex structures}

Our last illustration of classical instances of the concurrence relation concerns {\bf generalized complex structures} \cite{Gualtieri} --- that is, complex Dirac structures $L \subset \mathbb{T}_{\mathbb{C}}M$, which are transverse to their conjugate:
\begin{align*}
 L \oplus \overline{L} = \mathbb{T}_{\mathbb{C}}M.
\end{align*}

The description below of a generalized complex structure in terms of certain orthogonal endomorphisms $J$ of $\mathbb{T}M$ will be of use shortly:
\begin{lemma}\label{lem : GCS}
A complex Lagrangian subbundle $L \subset \mathbb{T}_{\mathbb{C}}M$ is transverse to its conjugate exactly when it is of the form
\begin{align*}
 L = \{J(a)+ia \ | \ a \in \mathbb{T}M\},
\end{align*}
where $J:\mathbb{T}M \to \mathbb{T}M$ is $\langle \cdot,\cdot\rangle$-orthogonal and squares to $-\mathrm{id}$. In the splitting $\mathbb{T}M = TM \oplus T^*M$, $J$ takes the form
\begin{align}\label{eq : J matrix}
 J = \begin{bmatrix}
      a & \pi^{\sharp} \\
      \omega^{\sharp} & -a^*
     \end{bmatrix}
\end{align}
where $a:TM \to TM$ is a linear map, $\omega \in \Omega^2(M)$ is a two-form and $\pi \in \mathfrak{X}^2(M)$ is a bivector, which are such that
\begin{align*}
 && a^2+\pi^{\sharp}\omega^{\sharp} = -\mathrm{id}, && a^*\omega^{\sharp}=\omega^{\sharp}a, && a\pi^{\sharp}=\pi^{\sharp}a^*.
\end{align*}
Moreover, $J$ is a generalized complex structure exactly when the following conditions are met:
\begin{enumerate}[C1)]
 \item $\pi$ is Poisson;
 \item the concomitant \eqref{eq : concomitant a pi} of $a$ and $\pi$ vanishes;
 \item $\mathrm{N}(a)(u,v) = -\pi^{\sharp}[u,v]^{\omega}$ for all $u,v \in \mathfrak{X}(M)$;
 \item the concomitant \eqref{eq : concomitant a omega} of $a$ and $\omega$ vanishes.
\end{enumerate}
\end{lemma}
\begin{proof}
 All of the claims are proven in \cite[Proposition 2.2]{Crainic}, except that in the stead of C4) we have in \emph{loc. cit.}
 \begin{align}
  \mathrm{d} \omega_a(u,v,w) & = \mathrm{d} \omega(a(u),v,w)+\mathrm{d} \omega(u,a(v),w)+\mathrm{d} \omega(u,v,a(w)) \label{eq : alt C4}
 \end{align}
 for all $u,v,w \in \mathfrak{X}(M)$, and where $\omega_a \in \Omega^2(M)$ is the two-form for which $\omega_a^{\sharp} = \omega^{\sharp}a$. Rewriting the LHS of \eqref{eq : alt C4} as  $\langle [u,v]^{\omega a},w\rangle$, and the RHS as
\begin{align*}
 \langle [a(u),v]^{\omega}+[u,a(v)]^{\omega}+a^*[u,v]^{\omega},w\rangle & = \langle [u,v]^{\omega a}-\mathrm{C}(a,\omega),w\rangle
\end{align*}
to conclude that \eqref{eq : alt C4} holds true exactly when $\mathrm{C}(a,\omega)$ vanishes identically.
\end{proof}

In a generalized complex structure, the Poisson structure $\pi$ is $a$-symmetric and has zero concomitant --- so $(\pi,a)$ need not be a $PN$-structure. Similarly, $(\omega,a)$ need not be a $\Omega N$-structure: not only $a$ need not be Nijenhuis, also $\omega$ need not be closed. As we show below, these conditions are simultaneously satisfied exactly when the generalized complex structure concurs with its conjugate:

\begin{proposition}[Concurrence of a generalized complex structure and its conjugate]
 For a generalized complex structure $L \subset \mathbb{T}_{\mathbb{C}}M$ corresponding
 \begin{align*}
  J = \begin{bmatrix}
      a & \pi^{\sharp} \\
      \omega^{\sharp} & -a^*
     \end{bmatrix},
 \end{align*}
 we have that
 \begin{align*}
  && L \star \mathcal{R}_{-1}(\overline{L}) = \mathrm{Gr}\left( \tfrac{1}{2i}\pi \right), && L \circledast \mathcal{R}_{-1}(\overline{L}) = \mathrm{Gr}\left( \tfrac{1}{2i}\omega \right).
 \end{align*}
 In particular,
 \begin{align*}
  && L \smile \overline{L} && \Longleftrightarrow && \mathrm{d} \omega = 0.
 \end{align*}
 Therefore for a generalized complex structure $J$ as above, the following conditions are equivalent:
 \begin{enumerate}[i)]
  \item $J$ concurs with its conjugate $\overline{J}=-J$;
  \item $(\pi,a)$ is a $PN$-structure and $(\omega,a)$ is a $\Omega N$-structure. 
 \end{enumerate}
\end{proposition}
\begin{proof}
Let $\sigma_{\pm}$ denote the complex-linear maps
\begin{align*}
 & \sigma_{\pm}: \mathbb{T}_{\mathbb{C}}M \to \mathbb{T}_{\mathbb{C}}M, & \sigma_{\pm}(x)=Jx\pm ix.
\end{align*}Then we have exact sequences
\begin{align*}
 & 0 \to \overline{L} \to \mathbb{T}_{\mathbb{C}}M \stackrel{\sigma_+}{\longrightarrow} L \to 0, & 0 \to L \to \mathbb{T}_{\mathbb{C}}M \stackrel{\sigma_-}{\longrightarrow} \overline{L} \to 0.
\end{align*}
In fact, because $L \cap \overline{L} = 0$, the restriction of $\sigma_{\pm}$ to $\mathbb{T}M$ gives isomorphisms of real vector bundles $L \stackrel{\sigma_+}{\longleftarrow} \mathbb{T}M \stackrel{\sigma_-}{\longrightarrow} \overline{L}$. Therefore any $(x,y) \in L \times \overline{L}$ is of the form
\begin{align*}
 x & =  \sigma_+(u+\xi) = (a(u)+\pi^{\sharp}(\xi)+iu)+(\omega^{\sharp}(u)-a^*(\xi)+i\xi), \\
 y & = \sigma_-(v+\eta) = (a(v)+\pi^{\sharp}(\eta)-iv)+(\omega^{\sharp}(v)-a^*(\eta)-i\eta)
\end{align*}
for a unique $(u+\xi,v+\eta) \in \mathbb{T}M^2$. Then $\mathrm{pr}_T(x-y)=0$ if and only if
\begin{align*}
 & u=-v, & 2a(u)=\pi^{\sharp}(\eta-\xi),
\end{align*}in which case
\begin{align*}
 x & = \left(\tfrac{1}{2}\pi^{\sharp}(\eta+\xi)+iu\right)+\left(\omega^{\sharp}(u)-a^*(\xi)+i\xi\right), \\ y & = \left(\tfrac{1}{2}\pi^{\sharp}(\eta+\xi)+iu\right)+\left(-\omega^{\sharp}(u)-a^*(\eta)-i\eta\right).
\end{align*}
Therefore
\begin{align*}
 x-\mathrm{pr}_{T^*}(y) = \left(\tfrac{1}{2}\pi^{\sharp}(\eta+\xi)+iu\right) + \left(2\omega^{\sharp}(u)+a^*(\eta-\xi)+i(\eta+\xi)\right)
\end{align*}lies in $\mathrm{Gr}(\tfrac{1}{2i}\pi)$, since
\begin{align*}
 \tfrac{1}{2i}\pi^{\sharp}(x-\mathrm{pr}_{T^*}(y)) & = \tfrac{1}{2i}\left(2\pi^{\sharp}\omega^{\sharp}(u)+\pi^{\sharp}a^*(\eta-\xi)+i\pi^{\sharp}(\eta+\xi)\right)\\
 & = \tfrac{1}{2i}\left(2(-u-a^2(u))+a\pi(\eta-\xi)+i\pi^{\sharp}(\eta+\xi)\right)\\
 & = \tfrac{1}{2i}\left(-2u+i\pi^{\sharp}(\eta+\xi)\right)\\
 & = \left(iu+\tfrac{1}{2}\pi^{\sharp}(\eta+\xi)\right).
\end{align*}
Therefore
\begin{align*}
 L \star \mathcal{R}_{-1}(\overline{L}) = \mathrm{Gr}(\tfrac{1}{2i}\pi).
\end{align*}
One proves symmetrically that
\begin{align*}
 L \circledast \mathcal{R}_{-1}(\overline{L}) = \mathrm{Gr}(\tfrac{1}{2i}\omega).
\end{align*}
Now observe that it follows from Proposition \ref{pro : concurrence pairwise and scalars} that
\begin{align*}
 && L \smile \overline{L} && \Longleftrightarrow && L \smile \mathcal{R}_{-1}(\overline{L}),
\end{align*}
and therefore $L$ and $\overline{L}$ concur exactly when $\mathrm{d} \omega = 0$. By comparison with the conditions C1)-C4) in the statement of Lemma \ref{lem : GCS}, we conclude that
\begin{align*}
 && L \smile \overline{L} && \Longleftrightarrow && (\pi,a) \ \text{is a $PN$-structure and} \ (\omega,a) \ \text{is a $\Omega N$-structure,}
\end{align*}
and this concludes the proof.
\end{proof}

\end{document}